\documentclass[letterpaper, 11pt]{amsart}

\usepackage{fullpage}
\usepackage{graphicx}
\usepackage{amsmath, amssymb, amsfonts, amsthm}
\usepackage{url}

\usepackage{tikz} 

\usepackage[ruled,vlined]{algorithm2e}

\usepackage{color}
\definecolor{darkblue}{rgb}{0,0,0.4}
\usepackage[plainpages=false,colorlinks=true,citecolor=blue,filecolor=black,linkcolor=red,urlcolor=darkblue]{hyperref}

\newtheorem{thm}{Theorem}[section]
\newtheorem{cor}[thm]{Corollary}
\newtheorem{prop}[thm]{Proposition}
\newtheorem{lem}[thm]{Lemma}
\newtheorem{conj}[thm]{Conjecture}
\theoremstyle{remark}
\newtheorem{rem}[thm]{Remark}

\newcommand{\ie}{{\it i.e.}}
\newcommand{\eg}{{\it e.g.}}
\newcommand{\N}{\mathbb N}
\newcommand{\R}{\mathbb R}

\newcommand{\calN}{\mathcal{N}}

\newcommand{\rad}{\mathrm{rad}}

\DeclareMathOperator*{\diam}{diam}

\graphicspath{{../figs/}} %note trailing /

\usepackage[backend=bibtex,firstinits=true,style=alphabetic,maxbibnames=9,natbib=true,url=false,sorting=nyt,doi=true,backref=false]{biblatex}
\renewbibmacro{in:}{\ifentrytype{article}{}{\printtext{\bibstring{in}\intitlepunct}}}

\begin{document}
\title{A maximal energy pointset \\ configuration problem}

\author{Braxton Osting}
\address{Department of Mathematics, University of Utah, Salt Lake City, UT}
\email{osting@math.utah.edu}
\thanks{BO was supported in part by U.S. NSF DMS 16-19755 and 17-52202.}

\author{Brian Simanek}
\address{Department of Mathematics, Baylor University, Waco, TX}
\email{Brian\_Simanek@baylor.edu}
%\thanks{}


\date{\today}

\begin{abstract}
We consider the extremal pointset configuration problem of maximizing a kernel-based energy subject to the geometric constraints that the points are contained in a fixed set, the pairwise distances are bounded below, and that every closed ball of fixed radius contains at least one point. We also formulate an extremal density problem, whose solution provides an upper bound for the pointset configuration problem in the limit as the number of points tends to infinity.  
Existence of solutions to both problems is established and the relationship between the parameters in the two problems is studied. Several examples are studied in detail, including the density problem for the $d$-dimensional ball and sphere, where the solution can be computed exactly using rearrangement inequalities. We develop a computational method for the density  problem that is very similar to the Merriman-Bence-Osher (MBO) diffusion-generated method. The method is proven to be increasing for all non-stationary iterations and is applied to study more examples. 
\end{abstract}

\maketitle

\section{Introduction} \label{s:Intro}
Optimal pointset configurations have broad applicability in physics and chemistry, information theory and communication,
and scientific computing. Typically, in such applications, one considers a pointset which \emph{minimizes} a certain energy. A prototypical example is the plum pudding model proposed by J. J. Thompson, where one seeks the positions of a fixed number of points (``electrons'') arranged on a sphere which minimizes the total electrostatic potential energy (as described by Coulomb's law). In this paper, we consider a pointset which \emph{maximizes} a certain energy, subject to constraints. 

For any measurable set $S\subseteq\R^p$, let $|S|_d$ denote the $d$-dimensional Hausdorff measure of $S$, where we normalize this measure so that the unit cube $U_d\subset\R^d\subset\R^p$ satisfies $|U_d|_d=1$.  In what follows, we will assume that $\Omega$ is an infinite and compact subset of $\R^p$ having finite and positive $d$-dimensional Hausdorff measure for some $d\in\{1,2,\ldots,p\}$.  In certain situations we will place stronger assumptions on the set $\Omega$, but we need only this minimal set of hypotheses to state the problem that we will study.

Define a \emph{kernel function}, $k\colon \Omega \times \Omega \to (0,\infty)$, which we will assume to be absolutely integrable and satisfy $k(x,y) = f(|x-y|)$ for some completely monotone function\footnote{A function, $f$, is \emph{completely monotone on $(0,\infty)$} if $f \in C^{\infty}(0,\infty)$ and $(-1)^\ell f^{(\ell)}(r) \geq 0$ for all $\ell \in \mathbb N_0$ and all $r>0$ \cite[Definition 7.1]{Wendland_2004}.}, $f\colon (0,\infty)\to [0,\infty)$.  We will also assume that $k$ is positive definite\footnote{A continuous kernel, $k\colon \Omega \times \Omega \to \mathbb R$, is \emph{positive definite} if, for all $N \in \mathbb N$, all sets of pairwise distinct centers $X = \{ x_1 \ldots, x_N \} \subset \mathbb R^d$, and all $\alpha \in \mathbb C^N \setminus \{0\}$, 
$\sum_{j=1}^N\sum_{k=1}^N \alpha_j \overline{\alpha_k} k(x_j,x_k) > 0$ 
\cite[Definition 6.24]{Wendland_2004}.}. 
Example kernels include 
\begin{itemize}
\item the \emph{Riesz $s$-kernel}, $k(x,y) = |x-y|^{-s}$, for $s \in(0,d)$,  
\item the \emph{exponential kernel}, $k(x,y) = \exp( -|x-y| / \sigma )$, for $\sigma > 0$, and
\item the \emph{Gaussian kernel}, $k(x,y) = (4 \pi \tau)^{-d/2} \exp(- |x-y|^2 / 4\tau)$, for $\tau >0$.
\end{itemize}

Fix  positive constants $r$ and $R$, and a positive integer $n$. 
In this paper, we will consider the problem of finding a collection of $n$ points, $X_n = \{ x_j \}_{j=1}^n \subset \Omega$, attaining the maximum in the optimization problem,   
\begin{subequations}
\label{e:DiscOpt}
\begin{align} 
\label{e:DiscOpt1}
\max_{\{x_k\}_{k=1}^n \subset \Omega} \ \   &  \frac{1}{2n^2} \sum_{i \neq j } k(x_i, x_j)  \\ 
\label{e:DiscOpt2}
\textrm{such that} \ \  
&| x_i - x_j | \geq rn^{-1/d},  \qquad\quad \ \  \forall  i \neq j \\ 
\label{e:DiscOpt3}
&X_n \cap \overline{B_{Rn^{-1/d}}(y)} \neq \emptyset,  \qquad \forall y \in \Omega. 
\end{align}
\end{subequations}
The distance here is the Euclidean distance in the ambient space, $\R^p$. 
Constraint \eqref{e:DiscOpt2} requires that the minimum pairwise distance is at least  $rn^{-1/d}$.  Constraint \eqref{e:DiscOpt3}  requires that every (closed) ball of radius $Rn^{-1/d}$ centered at a point in $\Omega$ contains a point of $X_n$. Of course, the domain $\Omega$ and constants $r$ and $R$ must be chosen so that there is at least one admissible configuration satisfying the constraints \eqref{e:DiscOpt2} and \eqref{e:DiscOpt3}. If an admissible configuration exists, then the upper semicontinuity of $k$ implies an extremal configuration satisfying \eqref{e:DiscOpt} exists, though it need not be unique. We discuss necessary conditions for the existence of admissible configurations for \eqref{e:DiscOpt} in Section \ref{s:AdPC}. 

To the best of our knowledge,  the maximization problem in \eqref{e:DiscOpt} has not been previously studied, although \citet[p. 226]{dragnev1997constrained} describe how such constraints can arise when one tries to incorporate varying conductivity into certain physical problems.  Our interpretation of  \eqref{e:DiscOpt} is as follows. The points represent distribution sites for a particular good or service. On one hand, \eqref{e:DiscOpt3} imposes the requirement that every location is within a distance $Rn^{-1/d}$ of a site. On the other hand, maybe because of distribution costs, it is cost effective to concentrate the sites together,  as described by the objective \eqref{e:DiscOpt1}. However, via \eqref{e:DiscOpt2}, we impose that the sites not be too close. Intuitively, the solution of \eqref{e:DiscOpt} will be to have the sites as closely packed as possible in the ``center of $\Omega$'', with enough sites arranged elsewhere so that \eqref{e:DiscOpt3} is satisfied. 
In this paper, we make this intuition precise.

\subsection{Outline and statement of results} 
We proceed as follows. 

In Section \ref{s:AdPC}, the  admissibility of pointset configurations in \eqref{e:DiscOpt} is studied. 
In Proposition \ref{p:adm}, we give sufficient conditions for $r$ and $R$  such that there exists an admissible pointset configuration for \eqref{e:DiscOpt} for large $n$.  
In Theorem \ref{t:weak}, we establish a preliminary result giving sufficient conditions so that  every weak-$*$ limit of the empirical  measure associated with the pointset is absolutely continuous with respect to $d$-dimensional Hausdorff measure.

In Section \ref{s:DensityFormulation}, we give a continuous analog of \eqref{e:DiscOpt} that, roughly speaking, corresponds to the density of the pointset in \eqref{e:DiscOpt} in the limit as $n \to \infty$. In Propositions \ref{p:contExist} and \ref{p:bang-bang}, the existence of a solution to this problem and properties of extremal densities are established. We then present results for several solvable examples. 

In Section  \ref{s:rel}, we discuss the relationship between the discrete and continuous problems. It is shown that the continuous problem gives an upper bound for the solution of the discrete problem; see Proposition \ref{p:precise} and Theorem \ref{t:equal}. 

In Section \ref{s:ex}, we introduce a computational method for the density  problem that is very similar to the Merriman-Bence-Osher (MBO) diffusion-generated method. 
In Proposition \ref{prop:AlgIncrease}, the method is proven to be increasing for all non-stationary iterations. The method is applied to several example problems, emphasizing qualitative properties of the resulting computed solutions. 

We conclude in Section \ref{s:disc} with a brief discussion.

\section{Admissibility of pointset configurations for \eqref{e:DiscOpt}} \label{s:AdPC}
For any given $n\in\N$, the possible values of $r$ and $R$ for which there exists an admissible configuration of $n$ points in $\Omega$ satisfying \eqref{e:DiscOpt2} and \eqref{e:DiscOpt3} depend on the geometry of $\Omega$ in a very delicate way.  However, we can find relatively simple constraints on $r$ and $R$ that are sufficient for the existence of an admissible configuration for all large $n\in\N$.

The key ideas that we will rely on are those of separation distance, covering radius, and mesh ratio.  To be precise, if $X_n=\{x_j\}_{j=1}^n\subset\Omega$, we define the \emph{separation distance}, $\delta(X_n)$, of this configuration by
\[
\delta(X_n)=\min_{i\neq j}|x_i-x_j|,
\]
where $|\cdot|$ denotes the Euclidean distance in $\R^p$.  Similarly, we define the \emph{covering radius}, $\eta(X_n)$, of this configuration by
\[
\eta(X_n)=\max_{y\in\Omega}\min_{x_j\in X_n}|y-x_j|
\]
and the \emph{mesh ratio}, $\gamma(X_n)$, of this configuration by
\[
\gamma(X_n)=\frac{\eta(X_n)}{\delta(X_n)}.
\]
Here we used the notation and terminology from \cite{bondarenko2014}.  The \emph{$n$-point best-packing distance on $\Omega$}, denoted $\delta_n(\Omega)$, is defined as the supremum of $\delta(X_n)$ over all subsets of $\Omega$ of cardinality $n$.  Any configuration that attains this supremum is called an \emph{$n$-point best-packing configuration on $\Omega$}.  Similarly, the \emph{$n$-point best-covering distance on $\Omega$},  denoted $\eta_n(\Omega)$, is defined as the infimum of $\eta(X_n)$ over all subsets of $\Omega$ of cardinality $n$.  Any configuration that attains this infimum is called an \emph{$n$-point best-covering configuration on $\Omega$}.  Now we can state conditions on $r$ and $R$ that guarantee the existence of admissible configurations.

\begin{prop}\label{p:adm}
Let $\Omega$ be an infinite compact subset of $\R^p$.  Suppose there exists $d\in\N$ and positive constants $C_*$ and $C^*$ so that
\[
\liminf_{n\to\infty}n^{1/d}\delta_n(\Omega)=C_* 
\qquad \textrm{and} \qquad 
\limsup_{n\to\infty}n^{1/d}\delta_n(\Omega)=C^*.
\]
If $r$ and $R$ satisfy $r<C_{_*}\leq C^*<R$, then for each sufficiently large $n\in\N$ there exists $X_n=\{x_j\}_{j=1}^n\subset\Omega$ that satisfies \eqref{e:DiscOpt2} and \eqref{e:DiscOpt3}.
\end{prop}

\begin{proof}
Since $r<C_*$, by definition we know that for each large $n\in\N$ there exists an $n$-point best-packing configuration $X_n^*\subset\Omega$ so that $X_n^*$ satisfies \eqref{e:DiscOpt2}.  Furthermore, \cite[Theorem 1]{bondarenko2014} assures us that we may choose $X_n^*$ to have mesh ratio at most $1$.  Thus
\[
\eta(X_n^*)\leq\delta(X_n^*)\leq(C^*+o(1))n^{-1/d}<Rn^{-1/d}
\]
when $n$ is sufficiently large.  We conclude that $X_n^*$ satisfies \eqref{e:DiscOpt3} when $n$ is sufficiently large, so it is an admissible configuration.
\end{proof}

Sets $\Omega$ for which the hypotheses of Proposition \ref{p:adm} are satisfied include smooth $d$-dimensional manifolds and certain perturbations of such sets (\eg, two intersecting line segments in $\R^2$) (see \cite[Section 1]{borodachov07}).  In the remainder of this paper, we'll assume that $r$ and $R$ in \eqref{e:DiscOpt} are chosen so that admissible configurations exist for all large $n$. 

The conditions \eqref{e:DiscOpt2} and \eqref{e:DiscOpt3} assure us that any sequence of configurations that satisfies these properties is well-distributed in the set $\Omega$.  To make this more precise, we need some additional terminology.

One idea we will need is that of $d$-dimensional packing premeasure; see \cite{tricot82}.  For a given set $\Omega$, we define $P_d(\Omega,\delta)$ by
\begin{align*}
&P_d(\Omega,\delta)=\sup\left\{\sum_{n=1}^N(\diam B_n)^d\colon {{B_i\mbox{ is a closed ball, $\diam(B_i)\leq\delta$},}\atop{\mbox{center}(B_i)\in \Omega, \ B_i\cap B_j=\emptyset\mbox{ when }i\neq j}}\right\},
\end{align*}
where $N\in\N\cup\{\infty\}$ in the above expression.  We then define the \emph{$d$-dimensional packing premeasure}, 
\[
P_d(\Omega)=\lim_{\delta\rightarrow0^+}P_d(\Omega;\delta),
\]
which is a premeasure in the sense of \cite[Definition 5]{rogers70} on the collection of totally bounded subsets of $\R^p$ (the limit as $\delta\rightarrow0^+$ exists by monotonicity).  We will also need  some notation associated with Hausdorff measure.  For each $\delta>0$, define 
\[
h_d(\Omega,\delta)=\inf\left\{\sum_{j=1}^N(\diam B_j)^d\colon \mbox{each }B_j\mbox{ is an open ball of diameter}<\delta, \ \Omega\subseteq\bigcup_{j=1}^NB_j\right\},
\]
although for compact sets in $\R^p$, this quantity remains unchanged if we consider closed balls instead of open balls. The \emph{$d$-dimensional Hausdorff outer measure}, $h_d$, is then
\[
h_d(\Omega)=\lim_{\delta\rightarrow0^+}h_d(\Omega,\delta).
\]
Now we can state our result.

\begin{thm}\label{t:weak}
Suppose there exists a basis for the topology on $\Omega$ consisting of open sets $\{U_{\alpha}\}_{\alpha\in I}$ that satisfy
\begin{itemize}
\item[i)] the boundary (in $\Omega$) $\partial U_{\alpha}$ of $U_{\alpha}$ has $h_d(\partial U_{\alpha})=P_d(\partial U_{\alpha})=0$ for all $\alpha\in I$,
\item[ii)] $P_d(U_{\alpha})=h_d(U_{\alpha})$ for all $\alpha\in I$.
\end{itemize}
For $n\in\{2,3,\ldots\}$, let $X_n\subset\Omega$ satisfy the conditions \eqref{e:DiscOpt2} and \eqref{e:DiscOpt3} and define the measure
\begin{align}\label{nun}
\nu_n:=\frac{1}{n}\sum_{j=1}^n\delta_{x_n}. 
\end{align}
Every weak-$*$ limit point of $\{\nu_n\}_{n\geq2}$ is mutually absolutely continuous with respect to $d$-dimensional Hausdorff measure on $\Omega$.
\end{thm}

\begin{proof}
Since $\Omega$ is compact, there exist weak-$*$ limit points of the sequence $\{\nu_n\}_{n\geq2}$.  Let $\nu$ be such a limit point and choose any $U_{\alpha}$ as in the statement of the theorem. First we will show that $\nu(\partial U_{\alpha})=0$.  Since $\partial U_{\alpha}$ is compact, we know that for any $\varepsilon>0$ we can find an open set $W$ that contains $\partial U_{\alpha}$, is a finite union of basis elements for the topology of $\Omega$, and satisfies $P(W)<\varepsilon$.  By condition \eqref{e:DiscOpt2}, we calculate
\[
\nu(\partial U_{\alpha})\leq \limsup_{n\to\infty}\nu_n(W)\leq\limsup_{n\to\infty}\frac{nP_d(W,rn^{-1/d})}{nr^d}=\frac{P_d(W)}{r^d}<\frac{\varepsilon}{r^d}.
\]
Since $\varepsilon>0$ was arbitrary, we conclude that $\nu(\partial U_{\alpha})=0$.

By similar reasoning, we calculate
\[
\nu(U_{\alpha})\leq \limsup_{n\to\infty}\nu_n(U_{\alpha})\leq\limsup_{n\to\infty}\frac{nP_d(U_{\alpha},rn^{-1/d})}{nr^d}=\frac{P_d(U_{\alpha})}{r^d}.
\]
%Now, for $\varepsilon>0$, let $\bar{U}_{\alpha}=\{x\in\Omega:\mbox{dist}(x,U_{\alpha})<\varepsilon\}$.
Using condition \eqref{e:DiscOpt3}, we calculate
\[
\nu(\bar{U}_{\alpha})\geq\liminf_{n\to\infty}\nu_n(\bar{U}_{\alpha})\geq \liminf_{n\to\infty}\frac{nh_d(\bar{U}_{\alpha},2Rn^{-1/d})}{n(2R)^d}=\frac{h_d(\bar{U}_{\alpha})}{(2R)^d}.
\]
Since $\nu(\partial U_{\alpha})=h_d(\partial U_{\alpha})=0$, it follows that
\[
\frac{h_d(U_{\alpha})}{(2R)^d}\leq\nu(U_{\alpha})\leq \frac{P_d(U_{\alpha})}{r^d}=\frac{h_d(U_{\alpha})}{r^d},
\]
which proves the claim for the sets $\{U_{\alpha}\}_{\alpha\in I}$.  The claim for a general open set now follows from the Monotone Convergence Theorem.
\end{proof}

\begin{rem}  In fact, the proof of Theorem \ref{t:weak} shows that the limiting measure $\nu$ has a density that we can control by choosing $r$ and $R$ appropriately.  We will return to this theme in Section \ref{s:rel}.
\end{rem}

Sets $\Omega$ that satisfy the conditions of Theorem \ref{t:weak} include the $d$-dimensional sphere, where the set $\{U_{\alpha}\}_{\alpha\in I}$ can be chosen as the set of all spherical caps. We'll revisit this example in Section \ref{s:DiscSphere}.

\section{Density formulation} \label{s:DensityFormulation}
Let $\Omega \subset \R^p$ and $k\colon \Omega \times \Omega \to (0,\infty)$ be a domain and kernel satisfying the assumption in Section~\ref{s:Intro}.  Define the quadratic functional $E\colon L^\infty(\Omega) \to \mathbb R$ by 
\begin{equation} \label{e:E}
E[\rho] := \frac{1}{2}  \int_{\Omega \times \Omega}  \ k(x,y)  \ \rho(x) \ \rho(y) \ dx \ dy,
\end{equation}
where $dx$ is $d$-dimensional Hausdorff measure on $\Omega$, normalized as in Section~\ref{s:Intro}.  We consider the optimization problem 
\begin{equation} \label{e:DensityFormulation}
\sup \ \{ E[\rho] \colon \rho \in A(\Omega, \rho_+, \rho_-) \},
\end{equation}
where, for constants $\rho_+$ and $\rho_-$ satisfying $\rho_+ \geq | \Omega|_d^{-1} \geq \rho_- > 0$,  the \emph{admissible class}, $A(\Omega, \rho_+,\rho_-)$, is defined as
\begin{equation} \label{e:A}
A(\Omega, \rho_+,\rho_-) :
= \{ \rho \in L^\infty(\Omega) \colon \int_\Omega \rho(x) \,dx=1, \ \  \rho_- \leq \rho(x) \leq \rho_+ \textrm{ for  a.e. } x \in \Omega \}.
\end{equation}
The admissible class is nonempty since it contains the constant function, $\rho \equiv |\Omega|_d^{-1}$.  In Section \ref{s:rel}, we will show that \eqref{e:DensityFormulation} is the density formulation of the discrete problem  \eqref{e:DiscOpt}. Here, we first establish some properties of \eqref{e:DensityFormulation}. 

\begin{prop} \label{p:contExist} Let $\Omega \subset \R^p$ and $k\colon \Omega \times \Omega \to (0,\infty)$ be a kernel function satisfying the assumptions in Section \ref{s:Intro}. 
Let  $\rho_+ \geq | \Omega|^{-1} \geq \rho_- > 0$. 
The supremum in \eqref{e:DensityFormulation} is attainted by at least one function $\rho^\star \in A(\Omega, \rho_+,\rho_-)$.
\end{prop}

\begin{proof} For every  $\rho \in A(\Omega, \rho_+, \rho_-)$, we have the lower bound  $E[\rho]\geq0$.  Using H\"older's inequality, we have that
$$
E[\rho] = \frac{1}{2}  \int_{\Omega \times \Omega}  \ k(x,y)  \ \rho(x) \ \rho(y) \ dx \ dy 
\leq \frac{1}{2} \| k \|_{L^1(\Omega \times \Omega)}  \| \rho\|_{L^\infty(\Omega)}^2
$$
and therefore, for $\rho_1, \rho_2 \in A(\Omega,\rho_+, \rho_-)$, 
$$
E[\rho_1 - \rho_2] \leq   \frac{1}{2} (\rho_+ - \rho_-) \| k \|_{L^1(\Omega \times \Omega)}   \| \rho_1 - \rho_2 \|_{L^\infty(\Omega)}. 
$$
It follows that $E$ is strongly continuous in the $L^\infty(\Omega)$ topology on $A(\Omega,\rho_+, \rho_-)$ and therefore continuous for the $weak-*$ topology. The result then follows from the weak-$*$ sequential compactness of the  admissible class $A(\Omega,\rho_+, \rho_-)$. 
\end{proof}

It is useful to define the integral operator $K\colon L^2(\Omega) \to L^2(\Omega)$, by
\begin{equation} \label{eq:HS}
(K \phi) (x) := \int_{\Omega} k(x,y)  \ \phi(y) \ dy. 
\end{equation}
Note that we can write $E[\rho] = \frac{1}{2} \langle \rho, K \rho\rangle_{L^2(\Omega)}$. The functional $E\colon L^\infty(\Omega) \to \mathbb R$ has a Fr\'echet derivative, $\delta E \colon L^\infty(\Omega) \to L(L^2(\Omega), \mathbb R) \cong L^2(\Omega)$,  given by 
\begin{equation} \label{eq:deriv}
\delta E \big|_\rho[\phi] = \Big\langle \frac{\delta E }{\delta \rho}\bigg|_\rho, \phi \Big\rangle  = \langle K \rho, \phi \rangle,  \qquad \phi \in  L^2(\Omega).
\end{equation}

\begin{prop} \label{p:bang-bang} 
Let $\Omega \subset \R^p$ and $k\colon \Omega \times \Omega \to (0,\infty)$ be a kernel function satisfying the assumptions in Section \ref{s:Intro}. 
Let  $\rho_+ \geq | \Omega|^{-1} \geq \rho_- > 0$. 
If $\rho^\star\in A(\Omega,\rho_+,\rho_-)$  is a local maximizer of $E$ on $A(\Omega, \rho_+,\rho_-)$,    then 
\begin{equation} \label{e:bang-bang}
\rho^\star(x) \in \{ \rho_- , \rho_+ \}  \quad \textrm{ for a.e. } x \in \Omega. 
\end{equation}
Furthermore,
\begin{equation} \label{e:KKT}
\rho^\star(x) = \begin{cases}
\rho_+ & \textrm{if } \  \ (K\rho^\star) (x) > \alpha^\star  \\
\rho_- & \textrm{if } \  \ (K\rho^\star) (x) < \alpha^\star 
\end{cases} 
\qquad \textrm{a.e.} \ x \in \Omega, 
\end{equation}
where $\alpha^\star$ is the  smallest value $\alpha$ such that $| \{ x \in \Omega \colon (K \rho^\star) (x) < \alpha \} |_d >  \frac{ \rho_+ - |\Omega|_d^{-1}  }{\rho_+ - \rho_-} |\Omega|_d $.
\end{prop}
\begin{proof}
 The functional $E$ satisfies  the identity 
$$
E[\theta \rho_1 + (1-\theta) \rho_2] = \theta E[\rho_1] + (1-\theta) E[\rho_2] 
- \theta (1-\theta) E[\rho_1 - \rho_2]. 
$$
Since $E$ is positive definite, for $\rho_1 \neq \rho_2$ and $\theta\in (0,1)$, we have that 
$$
E[\theta \rho_1 + (1-\theta) \rho_2] < \theta E[\rho_1] + (1-\theta) E[\rho_2], 
$$
which shows that $E$ is a strictly convex functional. It follows that only extremal points of $A(\Omega,\rho_+,\rho_-)$, which are functions of the form in \eqref{e:bang-bang}, can be local maximizers in \eqref{e:DensityFormulation}.

Assume $\rho^\star$ is a local maximum in \eqref{e:DensityFormulation} and suppose that \eqref{e:KKT} does not hold. 
Define the sets 
$$ 
B_+ = \{ x \in \Omega \colon \rho^\star(x) = \rho_+ \}
\qquad \textrm{and} \qquad  
B_- = \{ x \in \Omega \colon \rho^\star(x) = \rho_- \}, 
$$
where it is understood the equalities hold a.e.. 
 Then there exist disjoint, positive measure sets
$C_+ \subset \{ K\rho^\star < \alpha^\star \} \cap B_+$ and 
$ C_- \subset \{ K\rho^\star > \alpha^\star \} \cap B_- $ that satisfy $|C_+|_d = |C_-|_d$. 
For $\epsilon > 0$, consider the test function given by 
$$
g = \rho^\star + \varepsilon \left( 1_{C_-} - 1_{C_+} \right).  
$$ 
(Basically, we propose moving a small amount of density from $\{ K\rho^\star < \alpha \} $ to $\{ K \rho^\star > \alpha \}$, which is admissible since we are taking from $B_+$ and giving to $B_-$.) Note that $g\in A(\Omega,\rho_+,\rho_-)$ and we use strong convexity of $E$ to compute 
\begin{align*}
E[g] - E[\rho^\star] &> \langle K\rho^\star, g - \rho^\star \rangle \\
&=  \varepsilon \langle K\rho^\star, (1_{C_-} - 1_{C_+}) \rangle \\
&\geq  \varepsilon   \alpha (|C_-|_d - |C_+|_d ) \\
&= 0, 
\end{align*}
which contradicts the local optimality of $\rho^\star$. \end{proof}

The property in \eqref{e:bang-bang} that the optimal density attains  the allowed maximum and  minimum values almost everywhere  is sometimes referred to as the  ``bang-bang'' property of solutions. 
Many problems have similar structure, including problems involving the principle eigenvalue for composite materials; see \citep{krein1955certain,cox1990extremal,chanillo2008regularity}. 
%\citep{cox1990extremal2,chanillo2000symmetry,henrot2006extremum,chanillo2008weak,Osting:2012,kao2013efficient} 
We will retain the definition of the sets $B_{\pm}$ from the above proof (i.e. for an optimal density $\rho^{\star}$ for Problem \eqref{e:DensityFormulation}, the set $B_{\pm}$ will be the set where $\rho^{\star}=\rho_{\pm}$).

\subsection*{Example: constant kernel} Consider the situation when $k(x,y) = 1$.  Notice that this kernel does not satisfy the assumptions in Section \ref{s:Intro} because it is not positive definite.  Then for any $\rho \in A(\Omega,\rho_+,\rho_-)$ we compute
$$
E[\rho] =  \int_\Omega \rho(x) dx \ \int_\Omega \rho(y) dy = 1. 
$$
Thus all admissible densities have the same objective value, so the extremal solution is non-unique.

\bigskip

\subsection*{Example: delta distribution kernel} Consider $k(x,y) = \delta_0(x-y)$, where $\delta_0$ is the delta distribution. Note that this kernel does not satisfy the assumptions in Section \ref{s:Intro}. For any $\rho$ satisfying \eqref{e:bang-bang}, we compute 
$$
E[\rho] =  \int_\Omega \rho^2(x) dx 
=  \left( \rho_+^2 | \Omega_+|_d + \rho_-^2 |\Omega_- |_d \right) 
=  \rho_+ + \rho_-  - \rho_+ \rho_-|\Omega|_d. 
$$
Again all admissible densities have the same value, so the extremal solution is non-unique. 

\bigskip

\subsection{Solution of \eqref{e:DensityFormulation} for the $d$-dimensional ball} \label{s:ball} 
Let $\Omega=B_R(0)$ be the closed ball centered at $0$ with radius $R>0$. Here we use a rearrangement argument to show that the optimal density, $\rho^\star$, is spherically symmetric. The following theorem can be found in \cite[p.25]{kawohl1985rearrangements}, \cite[p.296]{friedman2010variational} and \cite[Theorem 14.8]{simon2011convexity} (see also \cite{kawohl1999symmetrization}).
 
\begin{thm}\label{t:SymRe} Let $f$, $g$, and $h$ be nonnegative functions in $\mathbb R^n$ and let $f^*$, $g^*$ $h^*$ be their spherically symmetric decreasing rearrangements, respectively. Then 
$$
\int_{\mathbb R^n} \int_{\mathbb R^n} f(x) g(y) h(x-y)  dx  dy \leq 
\int_{\mathbb R^n} \int_{\mathbb R^n} f^*(x) g^*(y) h^*(x-y)  dx  dy. 
$$
\end{thm}

\begin{prop} \label{p:ball} Let  $\Omega = B_R(0)$ for $R > 0$ and $k(x,y) = f(|x-y|)$, where $f$ is decreasing. The spherically decreasing density that satisfies \eqref{e:bang-bang} is optimal in \eqref{e:DensityFormulation}.   
\end{prop}
\begin{proof}
Let $\rho$ be an admissible weight, \ie, $\rho\in A(\Omega,\rho_+,\rho_-)$.  Notice that $|\{\rho\geq\rho_-\}|_d=|\Omega|_d$.

Define $\rho_{\Omega}(x):=\rho(x)1_{\Omega}(x)$ for all $x\in\mathbb R^d$ and notice that $f(|x|)$ is equal to its own spherically symmetric decreasing rearrangement.  It follows from Theorem \ref{t:SymRe} that for $\Omega = B_R(0)$, 
\begin{align*}
E[\rho] 
& =  \frac{1}{2}  \int_{\mathbb R^d \times \mathbb R^d}  \ k(x,y)  \ \rho_\Omega(x) \ \rho_\Omega(y) \ dx \ dy \\
& \leq \frac{1}{2}  \int_{\mathbb R^d \times \mathbb R^d}  \ k(x,y)  \ \rho_\Omega^*(x) \ \rho_\Omega^*(y) \ dx \ dy \\
& = \frac{1}{2}  \int_{\Omega \times \Omega}  \ k(x,y)  \ \rho^\star(x) \ \rho^\star(y) \ dx \ dy \\ 
& = E[\rho^\star]. 
\end{align*}
Notice that $\rho^\star\in A(\Omega,\rho_+,\rho_-)$ because $|\{\rho^\star\geq\rho_-\}|_d= |\{\rho\geq\rho_-\}|_d=|\Omega|_d$ and the symmetry of $\Omega$ implies $\rho^\star\geq\rho_-$ on all of $\Omega$.  It is trivial to see that $\rho^\star\leq\rho_+$ on all of $\mathbb R^d$.
\end{proof}

Taking $d=1$ so that  $\Omega$ is the unit interval, Proposition \ref{p:ball} states that an optimal density for \eqref{e:DensityFormulation} is given by  $\rho^\star = \rho_+$ on a centered interval and one can check by elementary reasoning that this is the unique maximizer as long as $\rho_->0$. This is illustrated in Figure \ref{fig:oneInterval}. Note that if $\rho_- = 0$, we wouldn't get a unique solution; the interval where $\rho^\star = \rho_+$ can be put anywhere in $\Omega$.

\begin{figure}[t]
\begin{center}

\begin{tikzpicture}[scale=1,thick]

% outside
\draw (-3,-3) rectangle (3,3);

% center
\filldraw[fill=blue!40!white, draw=black] (-1,-1) rectangle (1,1);
\draw (0,0) node {$\rho_+ \rho_+$};

%corners
\filldraw[fill=blue!10!white, draw=black] (-3,-3) rectangle (-1,-1);
\filldraw[fill=blue!10!white, draw=black] (-3,1) rectangle (-1,3);
\filldraw[fill=blue!10!white, draw=black] (1,1) rectangle (3,3);
\filldraw[fill=blue!10!white, draw=black] (1,-1) rectangle (3,-3);

\draw (2,2) node {$\rho_- \rho_-$};
\draw (2,-2) node {$\rho_- \rho_-$};
\draw (-2,2) node {$\rho_- \rho_-$};
\draw (-2,-2) node {$\rho_- \rho_-$};

%edges
\filldraw[fill=blue!20!white, draw=black] (-1,1) rectangle (1,3);
\filldraw[fill=blue!20!white, draw=black] (-1,-3) rectangle (1,-1);
\filldraw[fill=blue!20!white, draw=black] (1,-1) rectangle (3,1);
\filldraw[fill=blue!20!white, draw=black] (-3,-1) rectangle (-1,1);

\draw (0,2) node {$\rho_+ \rho_-$};
\draw (0,-2) node {$\rho_+ \rho_-$};
\draw (2,0) node {$\rho_+ \rho_-$};
\draw (-2,0) node {$\rho_+ \rho_-$};
\end{tikzpicture}

\caption{An illustration demonstrating that for the interval (one-dimensional ball), the region $\Omega_+ = \{x \colon \rho(x) = \rho_+ \}$ is given by a centered interval.}
\label{fig:oneInterval}
\end{center}
\end{figure}

\subsection{Solution of \eqref{e:DensityFormulation} for the unit sphere} \label{s:sphere} 
For the unit sphere, $\Omega=\{x\in\mathbb R^{d+1}\colon |x|=1\}$, we will show that $\rho^\star$ is equal to $\rho_+$ on a spherical cap and equal to $\rho_-$ on the compliment of this spherical cap.  By spherical cap, we mean a subset of the sphere that lies to one side of a hyperplane that intersects the sphere.  We use the following result. 
\begin{thm}[{\cite[Thm. 2]{Baernstein1976}}]\label{t:sphericalrearrangement}
Let $h$ be a nondecreasing, bounded, and measurable function on the interval $[-1,1]$.   Then for all $f,g\in L^1(\Omega)$,
\[
\int_{\Omega}\int_{\Omega}f(x)g(y)h(\langle x,y\rangle)dxdy\leq \int_{\Omega}\int_{\Omega}\tilde{f}(x)\tilde{g}(y)h(\langle x,y\rangle)dxdy,
\]
where $\tilde{f}$ and $\tilde{g}$ are the spherically increasing rearrangements of $f$ and $g$. 
\end{thm}

\begin{prop}\label{p:cap}
Suppose $k(x,y)=f(|x-y|)$ satisfies the assumptions in Section \ref{s:Intro}.  If $\Omega$ is the $d$-dimensional unit sphere, then the solution to \eqref{e:DensityFormulation} is given by $\rho^\star$, where $\rho^\star$ is equal to $\rho_+$ on a spherical cap and $\rho_-$ on the rest of $\Omega$.
\end{prop}

\begin{proof}
Assume first that $k(x,y)$ is bounded.
%By Proposition \ref{p:bang-bang} it suffices to identify the set $B_+ = \{x\in\Omega\colon \rho^\star(x)=\rho_+\}$.
If we write
\[
E[\rho^\star]=\frac{1}{2}\int_{\Omega\times\Omega}f(|x-y|)\rho^\star(x)\rho^\star(y)\, dx\, dy=\frac{1}{2}\int_{\Omega\times\Omega}f(\sqrt{2-2\langle x,y\rangle})\rho^\star(x)\rho^\star(y)\, dx\, dy,
\]
then the result is an immediate consequence of Theorem \ref{t:sphericalrearrangement}. 

If $k(x,y)$ is unbounded, let $f_m:=\min\{m,f\}$.  Then for any distribution $\rho$, we apply the above reasoning to show that
\[
\int_{\Omega\times\Omega}f_m(\sqrt{2-2\langle x,y\rangle})\rho(x)\rho(y)\, dx\, dy\leq \int_{\Omega\times\Omega}f_m(\sqrt{2-2\langle x,y\rangle})\rho^\star(x)\rho^\star(y)\, dx\, dy. 
\]
Taking $m\rightarrow\infty$ on both sides and applying Monotone Convergence proves the result.
\end{proof}

\subsection{Non-symmetry and non-uniqueness for solutions of  \eqref{e:DensityFormulation}} \label{s:Intervals} 
 We have already seen that the solution of  \eqref{e:DensityFormulation} does not necessarily preserve symmetries of $\Omega$ and is therefore not unique, but our counterexample required $\rho_-=0$, which is trivial in the sense that it means we were looking at the wrong set $\Omega$ (we should have been working on $B_+$).  Next we will provide an explicit example where the optimal solution does not preserve the symmetries of $\Omega$ even when $\rho_->0$. 
 
 For $b> a>0$, we consider the $d=1$ dimensional domain
$$\Omega = [-b, -a] \cup [a,b]. $$
This domain is symmetric with respect to the origin, but below we show that the optimal solution $\rho^\star$ is not symmetric for all kernels $k$. In particular, this shows that the solution to \eqref{e:DensityFormulation}  is not necessarily unique.

Consider the case with $\rho_+ = \frac{2}{3}$, $\rho_- = \frac{1}{3}$, $a=1$, $b=2$, and $k(x,y) = f(|x-y|)$, where 
$$
f(r) = \begin{cases} 2-r & r\in[0,2] \\ 0 & r > 2 \end{cases}. 
$$
This kernel is positive definite by \cite[Thm. 6.20]{Wendland_2004}. Due to the support of $f$, the integrand of the energy, the integral \eqref{e:E} vanishes on the region 
$$
\left(  [-b, -a] \times  [a,b] \right)  \ \bigcup \   \left( [a,b]   \times  [-b, -a]  \right)   \ \subset \ \Omega \times \Omega. 
$$
Therefore, the domain of integration for the energy is simply $ [-b, -a]^2 \ \cup \ [a,b]^2 \subset  \Omega^2$. By Proposition \ref{p:ball}, we know that we can take $B_+$ so that $B_+\cap[a,b]$ is an interval centered at $\frac{a+b}{2}$ and similarly for $B_+\cap[-b,-a]$.  It remains only to determine the length of those intervals, which we do by explicit calculation.
We  observe that $| B_+ |_1 = \frac{|\Omega|_1^{-1} -\rho_-}{\rho_+ - \rho_-} |\Omega|_1 = 1$. 
It follows that we can write $B_+\cap[a,b]=[\frac{a+b}{2}-t,\frac{a+b}{2}+t]$ and $B_+\cap[-b,-a]=[-b+t,-a-t]$ for some value of $t\in[0,(b-a)/2]$. 
Let $\rho_t$ denote the corresponding density. It is not difficult to show that 
$$
E[\rho_t] = \frac{185}{216} + \frac{10}{9}  \left( t - \frac{1}{4} \right)^2, \qquad t \in  [0, 1/2]. 
$$
This quadratic function takes a minimum at $1/4$ (corresponding to the symmetric solution). The maximum energy solution thus corresponds to the endpoints, $t=0$ and $t=1/2$. These correspond to taking $\rho = \rho_+$ on one interval and $\rho = \rho_-$ on the other interval.  Since both $t=0$ and $t=1/2$ attain the maximum, the maximum is not attained by a unique configuration.

\subsection{Solution of \eqref{e:DensityFormulation} in the limit $\rho_+ \to \infty$}  We consider the asymptotic limit of  \eqref{e:DensityFormulation} as $\rho_+ \to \infty$. 

\begin{prop} \label{p:weak}
Let $\Omega$ be a smooth $d$-dimensional manifold and suppose $k(x,y) = f(|x-y|)$ where $f$ is strictly decreasing and continuous on $[0,\infty)$ and $k$ satisfies the assumptions in Section \ref{s:Intro}.  Assume also that $x_0$ is the unique point in $\Omega$ that attains 
\begin{equation} \label{e:center}
\max_{y \in \Omega} \ \int_{\Omega} k(x,y) \ dx. 
\end{equation}
Then in the limit $\rho_+ \to \infty$, the unique weak-$*$ limit of optimal densities is $\rho_- + m  \delta(x_0)$, where $m = 1 - \rho_- |\Omega|_d$ (chosen such that $\int_\Omega \rho = 1$). 
\end{prop}

\begin{proof}
Let $\tau$ be a weak-$*$ limit point of the measures $(\rho_-\chi_{B_-}+\rho_+\chi_{B_+})dx$ as $\rho_+ \to \infty$.  It must be the case that $\tau$ is of the form $\rho_-dx+md\mu$ for some probability measure $\mu$.
We then compute 
$$
E[\rho_-dx + m  d\mu] = \rho_-^2 \int_{\Omega \times \Omega} f(|x-y|) \ dx dy 
+ 2 m \rho_- \int_{\Omega} k(x,y) \ dxd\mu(y)
+ m^2\int_{\Omega\times\Omega}k(x,y)d\mu(x)d\mu(y) . 
$$
The middle term is bounded above by a multiple of the expression in \eqref{e:center} and the far-right term is bounded above by $m^2f(0)$.  Furthermore, both bounds are attained if and only if $\mu$ is the point mass at $x_0$.  Note that by taking $B_+$ to be a small ball centered around $x_0$, we see that $\rho_- + m  \delta(x_0)$ is a weak-$*$ limit as $\rho_+\rightarrow\infty$ of densities in $A(\Omega,\rho_+,\rho_-)$.  The desired conclusion now follows from the weak-$*$ continuity of $E[\cdot]$, which is a consequence of the continuity of $f$.
\end{proof}

For the special case that $\Omega$  a $d$-dimensoinal ball, we showed in Proposition \ref{p:ball} that $B_+$ is a ball of prescribed radius centered in $\Omega$.  For other domains, it is tempting to think that as $\rho_+ \to \infty$, the optimal density might attain the value $\rho_+$ on a ball centered in $\Omega$. However, the following example shows this to be false. 

\subsection{Example: the ellipse}
We consider the $\varepsilon$-parameterized family of ellipses given by 
$$
\Omega_\varepsilon = \{  (x,y) \colon (1+\varepsilon) x^2 + (1+\varepsilon)^{-1} y^2 =1 \}. 
$$
Note that $|\Omega_\varepsilon |_2= \pi$, independent of $\varepsilon\geq0$. 
In Proposition \ref{p:ball}, we proved that there exists a value $r'$ (depending on $\rho_+$ and $\rho_-$) such that the optimal density for the ball, $\Omega_0$ is given by 
$$
\rho^\star(x) = \begin{cases} 
\rho_+ & \|x\| < r'  \\
\rho_- & \|x\| > r'. 
\end{cases}
$$ 
If $B(0,r') \subset \Omega_\varepsilon$, then this is an admissible density. 
We now ask whether it is possible for this $\rho^\star$ to be optimal for $\varepsilon \neq 0$? The optimality condition \eqref{e:KKT}  would require that a particular level set of 
\begin{align*}
K \rho^\star(x) &= \int_{\Omega_\varepsilon} f(|x-y|) \rho^\star(y)  \ dy \\
&= \int_{\Omega_0} f(|x-y|) \rho^\star(y) \ dy
+ \int_{\Omega_\varepsilon \setminus \Omega_0} f(|x-y|) \rho^\star(y) \ dy
- \int_{\Omega_0 \setminus \Omega_\varepsilon} f(|x-y|) \rho^\star(y) \ dy
\end{align*}
be independent of $\varepsilon$. But, this is false if $f$ is convex, decreasing, and positive. 
To see this, we observe that as $\varepsilon$ changes the change in the values of $K \rho^\star(x)$ at $x = (r',0)$ and $x = (0,r')$ have opposite sign. 
For $x = (r',0)$, the value of $K \rho^\star(x)$ is decreasing in $\varepsilon$ since
$$ 
\frac{f(1-r') + f(1+r')}{2} \geq f(1) \geq f \left( \sqrt{1+(r')^2} \right). 
$$
The first inequality follows  from the convexity of $f$ and the second from $f$ being decreasing. 
Similarly, for $x = (0,r')$, the value of $K \rho^\star(x)$ is increasing in $\varepsilon$. 
This shows that for all domains $\Omega_\varepsilon$ with $\varepsilon > 0$ sufficiently small, the region where $\rho^\star \equiv \rho_+$ is not a ball.  However, we believe that as $\rho_+ \to \infty$, the region where  $\rho^\star \equiv \rho_+$ converges to a shrinking ball.

\medskip

In light of the above examples and observations, we make the following conjecture.

\begin{conj}\label{c:convex}
Suppose $\Omega$ and $k$ satisfy the assumptions in Section \ref{s:Intro}.  Suppose also that $\Omega$ is convex and $\rho_->0$.  Then the optimal density for \eqref{e:DensityFormulation} is unique and $B_+$ is convex.
\end{conj}

Our next example shows that in general, convexity is not needed to deduce uniqueness of the optimal density.

\subsection{Example: the cross}  Let us consider the case when $\Omega\subset\R^2$ and is given by
\[
\Omega=\{(x,y):\sqrt{x^2+y^2}\leq1,\, xy=0\}. 
\]
This is a union of two one-dimensional manifolds with boundary. 
Let us also set $k(x,y)=f(|x-y|_{\mathcal{M}})$, where $|\cdot|_{\mathcal{M}}$ is the Manhattan metric and $f(r)$ is a decreasing convex function that is continuous on $(0,\infty)$.  Note that this kernel does not satisfy the assumptions in Section \ref{s:Intro}.  Nevertheless, the problem (\ref{e:DensityFormulation}) still makes sense for this choice of $k$, and we can find the optimizer.  We will assume $\rho_+>\rho_->0$.

Suppose $B_+^x$ is the intersection of $B_+$ with the $x$-axis and $B_+^y$ is the intersection of $B_+$ with the $y$-axis.  Invoking Theorem \ref{t:SymRe} and the fact that $f$ is decreasing, it is clear that we increase $E$ by concentrating $B_+^x$ and $B_+^y$ in centered intervals in their respective axes (we allow for the possibility that one of these intervals is empty).  It remains to figure out the length of these intervals.

 If we label $B_+^x$ as $[-t,t]_x$ and $B_+^y$ as $[-(S-t),S-t]_y$, then the energy of such a distribution can be expressed as
\begin{align*}
&2\rho_+^2\int_{-t}^{t}\int_{y}^tf(x-y)dxdy+2\rho_+^2\int_{t-S}^{S-t}\int_{y}^{S-t}f(x-y)dxdy+4\rho_+\rho_-\int_{-t}^t\int_t^1f(x-y)dxdy\\
&+4\rho_+\rho_-\int_{t-S}^{S-t}\int_{S-t}^1f(x-y)dxdy+4\rho_-^2\int_t^1\int_y^1f(x-y)dxdy+4\rho_-^2\int_{S-t}^1\int_{y}^1f(x-y)dxdy\\
&+2\rho_-^2\int_{-1}^{-t}\int_t^1f(x-y)dxdy+2\rho_-^2\int_{-1}^{t-S}\int_{S-t}^1f(x-y)dxdy+8\rho_+^2\int_{0}^t\int_{0}^{S-t}f(x+y)dxdy\\
&+8\rho_+\rho_-\int_{0}^t\int_{S-t}^1f(x+y)dxdy+8\rho_+\rho_-\int_{0}^{S-t}\int_{t}^1f(x+y)dxdy+8\rho_-^2\int_{t}^1\int_{S-t}^1f(x+y)dxdy.
\end{align*}
If we take two derivatives of this expression with respect to $t$ and simplify, we get
\begin{align*}
&(\rho_+-\rho_-)[8\rho_+(f(2t)-f(t)+f(2S-2t)-f(S-t))\\
&\quad+4\rho_-(3f(1+t)-f(1-t)+3f(1+S-t)-f(1-S+t)-2f(2t)-2f(2S-2t))],
\end{align*}
which is negative because $f$ is decreasing and convex and $0\leq t,S-t\leq1$.  From this and the symmetry of the problem, we see that energy is maximized when $B_+$ is the union of two perpendicular segments of equal length that intersect at the origin, which is the midpoint of each segment.  One can calculate that the length of these segments is determined by
\[
t=\frac{1-4\rho_-}{4(\rho_+-\rho_-)}. 
\]

If we extend this example to the union of the segments $[-1,1]$ in each of the coordinate axes in $n$-dimensions, then by examining pairs and triples of segments in the energy maximizing configuration, we see that $B_+$ for the optimal configuration is $n$ equal length segments that intersect at their midpoints, which is the origin.  Thus, in all these examples, the optimal density for \eqref{e:DensityFormulation} is unique, but $\Omega$ is not convex.

\section{Relationship between the discrete \eqref{e:DiscOpt} and continuous \eqref{e:DensityFormulation} problems} \label{s:rel}

In this section we make precise the relationship between the discrete \eqref{e:DiscOpt} and continuous \eqref{e:DensityFormulation} problems.  The first step in doing so is to make precise the relationship between $r$, $R$, $\rho_+$ and $\rho_-$.  In Proposition \ref{p:adm} and Theorem \ref{t:weak}, we have already seen constraints on $r$ and $R$ that are sufficient for there to exist admissible configurations and that every weak-$*$ limit of the empirical  measures is absolutely continuous with respect to $d$-dimensional Hausdorff measure.  Now we must be more precise.

For every $d\in\N$, define
\[
\beta_d=\frac{\pi^{d/2}}{\Gamma(d/2+1)},
\]
so that the volume of the ball of radius $r$ in $\R^d$ is $\beta_dr^d$.  Also, let $\Delta_d$ be the \emph{upper packing density of $\R^d$}, defined by
\begin{align*}
\Delta_d&:=\sup\left\{\lim_{t\rightarrow\infty}\frac{\sum_{j=1}^{\infty}|S_j\cap [-t,t]^d|_d}{(2t)^d}\right\},
\end{align*}
where the supremum is taken over all collections $\{S_j\}_{j\in\N}$ of non-overlapping spheres of unit radius in $\R^d$ such that the limit exists.  Similarly  let $\Theta_d$ be the \emph{lower covering density of $\R^d$}, defined by
\begin{align*}
\Theta_d&:=\inf\left\{\lim_{t\rightarrow\infty}\frac{\sum_{j=1}^{\infty}|S_j\cap [-t,t]^d|_d}{(2t)^d}\right\},
\end{align*}
where the infimum is taken over all collections $\{S_j\}_{j\in\N}$ of spheres of unit radius in $\R^d$ whose union is all of $\R^d$ such that the limit exists.  These packing and covering constants will be the key to establishing a relationship between the extremal problems \eqref{e:DiscOpt} and \eqref{e:DensityFormulation}.  In order to do so, we will assume (for convenience) that $\Omega$ satisfies certain regularity conditions.  To state these conditions, we must define some additional notation.  For any set $X\subset\R^p$ and any $\delta>0$, we define
\begin{align*}
Q_d(X,\delta):=\sup\left\{N\beta_d\delta^d: {N\in\N,\,{\{B_i\}_{i=1}^N\mbox{ is a collection of closed balls, $\rad(B_i)=\delta$},}\atop{\mbox{center}(B_i)\in X,B_i\cap B_j=\emptyset\mbox{ when }i\neq j}}\right\},
\end{align*}
and then define
\[
Q_d(X):=\limsup_{\delta\rightarrow0^+}Q_d(X,\delta).
\]
Similarly, we define
\begin{align*}
C_d(X,\delta):=\inf\left\{N\beta_d\delta^d: {N\in\N,\,\{B_i\}_{i=1}^N\mbox{ is a collection of closed balls}}\atop{\qquad\qquad\mbox{center}(B_i)\in X,\,\rad(B_i)=\delta,\,\,X\subseteq\bigcup_{j=1}^NB_j}\right\},
\end{align*}
and then define
\[
C_d(X):=\liminf_{\delta\rightarrow0^+}C_d(X,\delta).
\]
We will say that the set $\Omega$ is of \textit{Euclidean type} if for every open set $U\subseteq\Omega$ that satisfies $|U|_d=|\bar{U}|_d$, it holds that $Q_d(U)=Q_d(\bar{U})=|U|_d\Delta_d$ and $C_d(U)=C_d(\bar{U})=|U|_d\Theta_d$.  It is easy to verify that sets like the unit cube in $\R^d$, the the unit sphere in $\R^{d+1}$, and two tangent spheres in $\R^{d+1}$ are of Euclidean type.  Now we can state a relationship between $r$, $R$, $\rho_+$, and $\rho_-$.

\begin{prop}\label{p:precise}
Suppose $\Omega$ satisfies the hypotheses of Theorem \ref{t:weak} and is of Euclidean type.  Assume that $r$ and $R$ have been chosen so that an admissible configuration satisfying \eqref{e:DiscOpt2} and \eqref{e:DiscOpt3} exists for every sufficiently large $n\in\N$.  For each large $n\in\N$, let $X_n\subset\Omega$ be a collection having cardinality $n$ and satisfying \eqref{e:DiscOpt2} and \eqref{e:DiscOpt3}.  If $\rho_+$ and $\rho_-$ have been chosen so that $\rho_+r^d\beta_d\geq2^d\Delta_d$ and $\rho_-R^d\beta_d\leq\Theta_d$, then every weak-$*$ limit point of the the measures $\{\nu_n\}_{n\geq2}$ defined in analogy with \eqref{nun} has density in $A(\Omega,\rho_+,\rho_-)$.
\end{prop}

\begin{proof}
Let $\nu$ be a weak-$*$ limit point of the measures $\{\nu_n\}_{n\geq2}$ and let $U$ be an open set in $\Omega$ that satisfies $|U|_d=|\bar{U}|_d$.  Observe that $\nu(\bar{U}\setminus U)=0$ because $\partial U$ has $d$-dimensional Hausdorff measure $0$ and by Theorem \ref{t:weak} we know that $\nu$ is mutually absolutely continuous with $d$-dimensional Hausdorff measure.

Notice that the collection of closed balls of radius $rn^{-1/d}/2$ centered at points of $X_n\cap U$ are disjoint.  Therefore,
\[
\nu(U)\leq \limsup_{n\rightarrow\infty}\nu_n(U)\leq\limsup_{n\to\infty}\frac{2^dQ_d(U,rn^{-1/d}/2)}{r^d\beta_d}\leq\frac{2^d\Delta_d|U|_d}{r^d\beta_d}\leq\rho_+|U|_d. 
\]

Now for each fixed $\varepsilon>0$, let $U_{\varepsilon}$ be an open set that satisfies $|U_{\varepsilon}|_d=|\bar{U}_{\varepsilon}|_d$, contains an $\varepsilon$-neighborhood of $\bar{U}$, and is contained in a $2\varepsilon$-neighborhood of $\bar{U}$.  Notice that the collection of closed balls of radius $Rn^{-1/d}$ centered at points of $X_n\cap U_{\varepsilon}$ cover $U_{\varepsilon/3}$ when $n$ is sufficiently large.  Therefore,
\[
\nu(U)\geq \liminf_{n\rightarrow\infty}\nu_n(U_{\varepsilon})\geq\liminf_{n\to\infty}\frac{C_d(U_{\varepsilon/3},Rn^{-1/d})}{R^d\beta_d}\geq\frac{\Theta_d|U_{\varepsilon/3}|_d}{R^d\beta_d}\geq\rho_-|U_{\varepsilon/3}|_d. 
\]
Therefore, by taking $\varepsilon\rightarrow0$ we see that
\[
\rho_-|U|_d\leq\nu(U)\leq\rho_+|U|_d, 
\]
as desired.  The result for a general open set now follows from the Monotone Convergence Theorem.
\end{proof}

Here is our main result of this section.

\begin{thm}\label{t:equal}
Suppose $\Omega$ is a $d$-dimensional smooth manifold that satisfies the hypotheses of Theorem \ref{t:weak} and is of Euclidean type and that $k$ satisfies the assumptions in Section \ref{s:Intro}.  Suppose also that $r$, $R$, $\rho_+$, and $\rho_-$ have been chosen so that admissible configurations  satisfying \eqref{e:DiscOpt2} and \eqref{e:DiscOpt3} exist for every sufficiently large $n\in\N$ and so that $\rho_+r^d\beta_d\geq2^d\Delta_d$ and $\rho_-R^d\beta_d\leq\Theta_d$.  Writing $k(x,y)=f(|x-y|)$,  assume that there is a constant $C_0$ so that 
\begin{equation}\label{cstar}
\sup_{{x\geq rn^{-1/d}}\atop{n\in\N}}\,\frac{f(x)}{f(x+\sqrt{2}Rn^{-1/d})}\leq C_0.
\end{equation}
Then
\begin{equation}\label{e:dtoc1}
\limsup_{n\to\infty} \ \{\mathrm{Problem }\,\, \eqref{e:DiscOpt} \} 
\ \ \leq \ \ 
\sup \ \{ E[\rho] \colon \rho \in A(\Omega, \rho_+, \rho_-) \}. 
\end{equation}
If we further assume that there is a sequence of $n$-point configurations $\{X_n'\}_{n=2}^{\infty}$ satisfying \eqref{e:DiscOpt2} and \eqref{e:DiscOpt3} and so that the measures $\{\nu_n'\}_{n\geq2}$ defined in analogy with \eqref{nun} converge in the weak-$*$ topology to a distribution with density $\rho^\star$ that is extremal for Problem \eqref{e:DensityFormulation}, then the inequality \eqref{e:dtoc1} is an equality with the $\limsup$ replaced by the full limit.
\end{thm}

It is easy to see that for functions like $f(t)=t^{-s}$ for some $s\in(0,d)$, $f(t)=-\log(t)$, or $f(t) = e^{-t/\sigma}$ for some $\sigma>0$, there exists a constant $C_0$ such that \eqref{cstar} holds, but that it does not hold for the function $f(t)=e^{t^{-2}}$.  The reason we make the assumption \eqref{cstar} is because of the following lemma.

\begin{lem}\label{possassump}
For each $n\geq2$, suppose $X_n=\{x_j\}_{j=1}^n\subset\Omega$ is a collection satisfying \eqref{e:DiscOpt2} and \eqref{e:DiscOpt3}.  Let $V_{ij}$ be the Voronoi cell of $(x_i,x_j)\in\Omega\times\Omega$ and let $k_{ij}$ be the average value of $k$ over $V_{ij}$.  If $k(x,y)$ satisfies the assumptions in Section \ref{s:Intro}, then
\[
\sup_{i\neq j}\frac{k(x_i,x_j)}{k_{ij}}\leq \sup_{{x\geq r/n^{1/d}}\atop{n\in\N}}\,\frac{f(x)}{f(x+\sqrt{2}Rn^{-1/d})}. 
\]
\end{lem}

\begin{proof}
Suppose $x_i\neq x_j$ in $X_n$ are given.  Then the condition \eqref{e:DiscOpt3} implies that the point in $V_{ij}$ furthest from the diagonal in $\Omega\times\Omega$ is a distance at most $|x_i-x_j|+\sqrt{2}Rn^{-1/d}$ from the diagonal.  Therefore,
\[
\frac{k(x_i,x_j)}{k_{ij}}\leq\frac{f(|x_i-x_j|)}{f(|x_i-x_j|+\sqrt{2}Rn^{-1/d})},
\]
which is upper bounded by
\[
\sup_{{x\geq r/n^{1/d}}\atop{n\in\N}}\,\frac{f(x)}{f(x+\sqrt{2}Rn^{-1/d})}.
\]
as desired.
\end{proof}

\begin{proof}[Proof of Theorem \ref{t:equal}]
For each $n\geq2$, let $X_n^*=\{x_j^*\}_{j=1}^n\subset\Omega$ be a configuration that satisfies \eqref{e:DiscOpt2} and \eqref{e:DiscOpt3} and define $\nu_n^*$ in analogy with \eqref{nun}.  Let $\calN\subseteq\N$ be a subsequence so that
the measures $\{\nu_n^*\}_{n\geq2}$ converge to a weak-$*$ limit $\nu$ as $n\to\infty$ through $\calN$.   By Proposition \ref{p:precise} the conditions \eqref{e:DiscOpt2} and \eqref{e:DiscOpt3} assure us that $\nu$ is mutually absolutely continuous with respect to $d$-dimensional Hausdorff measure on $\Omega$ and with density in $A(\Omega,\rho_+,\rho_-)$.

For any $\varepsilon>0$, let $\Lambda$ be a compact subset of $\Omega\times \Omega$ that does not intersect the diagonal of $\Omega\times\Omega$, has boundary with $d$-dimensional Hausdorff measure zero, is symmetric, and satisfies $|(\Omega\times\Omega)\setminus\Lambda|_d<\varepsilon$. We write
\begin{align}\label{lambdasum}
\frac{1}{2n^2}\sum_{i\neq j}k(x_i^*,x_j^*)&=\frac{1}{2}\sum_{(x_i,^*x_j^*)\in\Lambda}\frac{k(x_i^*,x_j^*)}{n^2}+\frac{1}{2}\sum_{{(x_i^*,x_j^*)\not\in\Lambda}\atop{i\neq j}}\frac{k(x_i^*,x_j^*)}{n^2}. 
\end{align}
The first sum on the right-hand side of \eqref{lambdasum} is equal to
\[
\frac{1}{2}\int_{\Lambda}k(x,y)d(\nu_n\times\nu_n)=\frac{1}{2}\int_{\Omega\times\Omega}1_{\Lambda}(x,y)k(x,y)d(\nu_n\times\nu_n)\rightarrow\frac{1}{2}\int_{\Lambda}k(x,y)\,d\nu(x)\,d\nu(y)
\]
as $n\to\infty$ through $\calN$.  To deal with the second sum on the right-hand side of \eqref{lambdasum}, we use Lemma \ref{possassump} to bound it from above by
\[
\frac{C_0}{2}\sum_{{(x_i^*,x_j^*)\not\in\Lambda}\atop{i\neq j}}\frac{k_{ij}}{n^2}=\frac{C_0}{2}\sum_{{(x_i^*,x_j^*)\not\in\Lambda}\atop{i\neq j}}\frac{k_{ij}}{n^2|V_{ij}|_{2d}}|V_{ij}|_{2d}. 
\]
From \eqref{e:DiscOpt2} and \eqref{e:DiscOpt3}, we know that $n^2|V_{ij}|_{2d}$ is bounded above and below by positive constants.  Therefore, we can bound this sum from above by an absolute constant multiplied by the integral of $k(x,y)$ over the union of the Voronoi cells associated to pairs $(x_i^*,x_j^*)\not\in\Lambda$. This can be made arbitrarily small by choosing $\varepsilon$ sufficiently small ($\Lambda$ large enough).  Therefore,
\begin{equation}\label{e:limsup}
\lim_{{n\to\infty}\atop{n\in\mathcal{N}}}\frac{1}{2n^2}\sum_{i\neq j}k(x_i^*,x_j^*)=\frac{1}{2}\int_{\Omega\times\Omega}k(x,y)\,d\nu(x)\,d\nu(y)\leq\sup \ \{ E[\rho] \colon \rho \in A(\Omega, \rho_+, \rho_-) \} . 
\end{equation}
Since this is true for every sequence $\{X_n^*\}_{n\geq2}$ of admissible configurations and every subsequence $\mathcal{N}$, this gives us the inequality we wanted.

To prove the reverse inequality, we assume that for each $n\geq2$ we may choose a configuration $X_n'=\{x_j'\}_{j=1}^n\subset\Omega$ as in the statement of the theorem.  Let $k_1$ be a continuous function on $\Omega\times\Omega$ satisfying $0\leq k_1\leq k$.  If $X_n^*$ is optimal for Problem \eqref{e:DiscOpt}, then we have
\begin{align*}
\frac{1}{2n^2}\sum_{i\neq j}k(x_i^*,x_j^*)&\geq\frac{1}{2n^2}\sum_{i\neq j}k(x_i',x_j')\geq\frac{1}{2n^2}\sum_{i\neq j}k_1(x_i',x_j')\\
&=\frac{1}{2}\int\int_{\Omega\times\Omega}k_1(x,y)d\nu_n'(x)d\nu_n'(y)-\frac{1}{2n^2}\sum_{j=1}^nk_1(x_j',x_j')\\
&\rightarrow\frac{1}{2}\int\int_{\Omega\times\Omega}k_1(x,y)\rho^\star(x)\rho^\star(y)\,dx\,dy
\end{align*}
as $n\to\infty$.  Taking the supremum over all such functions $k_1$ gives,
\begin{equation}\label{e:dtoc}
\liminf_{n\to\infty}\{\mbox{Problem }\, \eqref{e:DiscOpt} \} 
\ \ \geq \ \ 
\sup \ \{ E[\rho] \colon \rho \in A(\Omega, \rho_+, \rho_-) \} 
\end{equation}
Equations \eqref{e:limsup} and \eqref{e:dtoc} give the desired equality.
\end{proof}

The proof of Theorem \ref{t:equal} yields the following corollary.

\begin{cor}\label{c:weakstar}
Assume the hypotheses of Theorem \ref{t:equal}, including the existence of the sequence $\{X_n'\}_{n\geq2}$.  For each $n\in\{2,3,\ldots\}$, let $X_n^*=\{x_j^*\}_{j=1}^n\subset\Omega$ be a configuration that is optimal for Problem \eqref{e:DiscOpt} and define $\nu_n^*$ in analogy with \eqref{nun}.  Every weak-$*$ limit of the measures $\{\nu_n^*\}_{n\geq2}$ is extremal for Problem \eqref{e:DensityFormulation}.
\end{cor}

The only shortcoming of Theorem \ref{t:equal} is the assumption required to make the inequality \eqref{e:dtoc1} into an equality.  It is possible that such a sequence of configurations does not exist.  Indeed, the proof of Proposition \ref{p:precise} shows that if $\rho_+r^d\beta_d>2^d\Delta_d$ or $\rho_-R^d\beta_d<\Theta_d$, then no such sequence exists. However, if $\rho_+r^d\beta_d=2^d\Delta_d$ and $\rho_-R^d\beta_d=\Theta_d$ and one has certain additional information, then one can deduce the existence of the desired configurations $\{X_n'\}_{n\geq2}$.  We illustrate this with an example.

\subsection{Example: constructing $X_n'$} \label{s:DiscSphere}
Suppose $\Omega$ is of Euclidean type and $\rho_{\pm}$ have been chosen so that there is a solution to  Problem \eqref{e:DensityFormulation} for which $B_+$ is an open rectifiable\footnote{We refer the reader to \cite{borodachov07} for the definition of a rectifiable set.} set with smooth boundary satisfying $|B_+|_d=|\overline{B}_+|_d$.  Propositions \ref{p:cap} and \ref{p:ball} tell us that this is true if we take $\Omega=S^d\subset\R^{d+1}$ or $\Omega=B_1(0)\subseteq\R^{d}$.  Assume $r$ and $R$ have been chosen so that there exist configurations satisfying \eqref{e:DiscOpt2} and \eqref{e:DiscOpt3} for all large $n\in\N$ and also so that $R>2r$.  Suppose $\rho_+$ and $\rho_-$ are chosen so that $\rho_+r^d\beta_d=2^d\Delta_d$ and $\rho_-R^d\beta_d=\Theta_d$.

Fix one $\Omega_+$ satisfying $|\Omega_+|_d=|\overline{\Omega}_+|_d$ so that the density
\[
\rho^\star(x)=
\begin{cases}
\rho_+\quad & \quad x\in B_+\\
\rho_- & \quad x\in B_-=\Omega\setminus B_+
\end{cases}
\]
is extremal for Problem \eqref{e:DensityFormulation}.  For any $n\in\N$, let
\[
V_+(n):=\left\{x\in B_+:\mathrm{dist}(x,\partial B_+)\geq \frac{rn^{-1/d}}{2}\right\} 
\quad \textrm{and} \quad V_-(n):=\left\{x\in B_-:\mathrm{dist}(x,\partial B_+)\geq \frac{rn^{-1/d}}{2}\right\}. 
\]
We will also assume that for all sufficiently large $n$ and $m$, the set $V_{-}(n)$ admits an $m$-point best packing configuration that has mesh ratio equal to
\begin{equation}\label{e:mesh}
\frac{1}{2}\left(\frac{\Theta_d}{\Delta_d}\right)^{1/d}.
\end{equation}
According to \cite[Theorem 4]{bondarenko2014}, such a mesh ratio is the best one could possibly hope for, at least in an asymptotic sense.  We will discuss the practicality of this assumption later, but for now let us proceed with our construction.

For any $n\in\N$, define
\begin{equation}\label{e:nplus}
n_+=n\rho_+|B_+|_d-o(n)
\end{equation}
for a sequence $o(n)$ that we will specify later, and let $X_+(n)\subseteq V_+(n)$ be an $n_+$-point best-packing configuration of $V_+(n)$ that has mesh ratio at most $1$ (we used \cite[Theorem 1]{bondarenko2014}).  Now define $n_-:=n-n_+$ and let $X_-(n)\subseteq V_-(n)$ be an $n_-$-point best-packing configuration of $V_-(n)$ that has mesh ratio equal to the quantity in \eqref{e:mesh}.  We claim that $\{X_+(n)\cup X_-(n)\}_{n=N}^{\infty}$ (for some $N\in\N$) is a sequence of admissible configurations whose counting measures converge to the density $\rho^\star$ as $n\to\infty$.

First let us consider the admissibility of the configuration $X_+(n)\cup X_-(n)$ for large $n$.  Notice that by \cite[Equation 2.3]{borodachov07} it holds that
\begin{equation}\label{e:delr}
\delta_{n_+}(V_+(n))\sim2\left(\frac{\Delta_d}{\beta_d}\right)^{1/d}|B_+|_d^{1/d}n_+^{-1/d}\sim rn^{-1/d},
\end{equation}
where we used the assumed relationship between $\rho_+$, $\Delta_d$, $\beta_d$, and $r$.  Therefore, one may choose the sequence $o(n)$ in \eqref{e:nplus} appropriately so that $\delta_{n_+}(V_+(n))\geq rn^{-1/d}$ for all sufficiently large $n$.  Thus, the fact that $R>2r$ assures us that the set $X_+(n)$ is admissible on $\Omega\setminus V_-(n)$ (we used the assumption on the mesh ratio of $X_+(n)$ here).

Similar reasoning shows that
\[
\delta_{n_-}(V_-(n))\sim rn^{-1/d}\left(\frac{\rho_+}{\rho_-}\right)^{1/d},
\]
so $X_-(n)$ satisfies \eqref{e:DiscOpt2} on $V_-(n)$ and with enough room to spare to accommodate the additional $o(n)$ points from \eqref{e:nplus}.  To show that $X_-(n)$ satisfies \eqref{e:DiscOpt3} on $V_-(n)$, we calculate
\[
\eta(X_-(n))\leq\frac{1}{2}\left(\frac{\Theta_d}{\Delta_d}\right)^{1/d}\delta(X_-(n))=\frac{1}{2}\left(\frac{\Theta_d}{\Delta_d}\right)^{1/d}rn^{-1/d}\left(\frac{\rho_+}{\rho_-}\right)^{1/d}\sim Rn^{-1/d}. 
\]
Therefore, if we choose the sequence $o(n)$ in \eqref{e:nplus} appropriately, it will be true that $\eta(X_-(n))<Rn^{-1/d}$ for all large $n$.  We conclude that $X_+(n)\cup X_-(n)$ is admissible when $n$ is large.

It remains to consider the weak-$*$ limits of the counting measures.  It is clear by construction that any weak-$*$ limit $\nu$ satisfies $\nu(B_+)=\rho_+|B_+|_d$ and $\nu(B_-)=\rho_-|B_-|_d$. Thus it suffices to show that $\nu$ is uniform on $B_+$ and $B_-$.  This follows from \cite[Theorem 2.2]{borodachov07}.

\medskip

In the previous example, the assumption that $V_-(n)$ admits best packing configurations with mesh ratio \eqref{e:mesh} was essential in calculating the covering radius of the set $X_-(n)$ in $V_-(n)$.  The proof of \cite[Theorem 4]{bondarenko2014} shows that in general, one cannot hope to find a configuration with a smaller mesh ratio than \eqref{e:mesh} and in fact \cite[Theorem 5]{bondarenko2014} shows that subtleties arise even when considering nice sets like $S^2$.  However, if $d=2$, the fact that the best packing configuration in $\R^2$ and the best covering configuration in $\R^2$ are both given by the vertices of the equilateral triangle lattice, one could hope to attain the bound \eqref{e:mesh} in situations when $d=2$ and $B_+$ has a sufficiently regular boundary.  In general, it is difficult to prove such regularity results on the boundary of $B_+$, but we will return to this topic in Section \ref{s:ex} with some computational examples that suggest this phenomenon occurs quite often.

\subsection{Example: the interval, $[-1,1]$}  
In the case of  the interval $[-1,1]$, many of the quantities that we have so far discussed abstractly can be made explicit.  In this setting, an admissible configuration exists for all $n$ if and only if $r\leq2$ and $1\leq R$.  To give us some flexibility in our configurations, let us suppose that both of these inequalities are strict.  The extremal density $\rho^\star$ is equal to $r^{-1}$ on an interval of length $\frac{2 r (R-1) }{2R-r}$ centered around $0$ and equal to $(2R)^{-1}$ on the remainder of the interval. 

To gain some insight into what an optimal solution to (\ref{e:DiscOpt}) looks like, consider the case $n=4$, where the optimal solution can be computed explicitly. Indeed, assume $k(x,y)=f(|x-y|)$ for some decreasing and convex function $f$ that is continuous on $(0,\infty)$.  In this case, the covering bound involving $R$ implies that no two nearest neighbors can have separation exceeding $2R/n$, so the existence of an admissible configuration requires $r\leq8/3$ and $R\geq1$.  Suppose $x_1<x_2<x_3<x_4$ and define
\[
\tau:=\min\{|x_1-x_2|,|x_2-x_3|,|x_3-x_4|\}. 
\]
By symmetry, we may assume without loss of generality that $|x_1-x_2|\leq|x_3-x_4|$.  Suppose that $|x_2-x_3|>\tau$.  Then we must have $|x_1-x_2|=\tau$.  However, reflecting $x_2$ about the point $(x_1+x_3)/2$ gives us a new configuration, where the collection of pairwise distances between the points is the same, except $|x_2-x_4|$ has decreased.  Thus, this is an energy increasing transformation and so in the extremal configuration, we must have $|x_2-x_3|=\tau$.

If $8\leq3r+2R$, then one can check that there is an admissible configuration with all nearest neighbor distances equal to $\tau=r/n$.  If $8>3r+2R$, then we can increase the energy by sliding the points $x_1$ and $x_4$ towards $0$ (and moving $x_2$ and $x_3$ accordingly) so the extremal configuration satisfies $x_1=-1+R/4$ and $x_4=1-R/4$.  All that remains is to determine $|x_1-x_2|$ and $|x_3-x_2|$.

Recall we are assuming that $|x_1-x_2|\leq|x_3-x_4|$.  Suppose this inequality is strict.  If we slide $x_2$ and $x_3$ toward $x_1$ by an amount $\varepsilon>0$ that is very small, then the total change in the energy is
\begin{align*}
& f(x_2-x_1-\varepsilon)+f(x_2-x_1+\tau-\varepsilon)+f(x_4-x_3+\varepsilon)+f(x_4-x_3-\tau+\varepsilon)\\
&\qquad\qquad -f(x_2-x_1)-f(x_2-x_1+\tau)-f(x_4-x_3)-f(x_4-x_3-\tau)\\
&=[f(x_2-x_1-\varepsilon)+f(x_4-x_3+\varepsilon)-f(x_4-x_3)-f(x_2-x_1)]\\
&\qquad\qquad+[f(x_2-x_1+\tau-\varepsilon)+f(x_4-x_3-\tau+\varepsilon)-f(x_2-x_1+\tau)-f(x_4-x_3-\tau)]\\
&\geq\varepsilon[|f'(x_2-x_1)|+|f'(x_2-x_1+\tau)|-|f'(x_4-x_3)|-|f'(x_4-x_3+\tau)|]>0.
\end{align*}
Therefore, if $|x_1-x_2|<|x_3-x_4|$, then we can increase the energy by sliding $x_2$ and $x_3$ closer to $x_1$.  If $|x_1-x_2|=|x_3-x_4|$, then the convexity of $f$ shows that sliding $x_2$ and $x_3$ toward $x_1$ by a small and equal amount is an energy increasing move.  We conclude that to maximize the energy, $x_2$ should be as close to $x_1$ as possible within the constraint of admissibility.

Using these ideas, we can determine an optimal configuration for every choice of $r,R$ satisfying $r\leq8/3$ and $R\geq1$, which we summarize with the following results.

\begin{figure}[t!]
\begin{center}
\includegraphics[width=.7\textwidth]{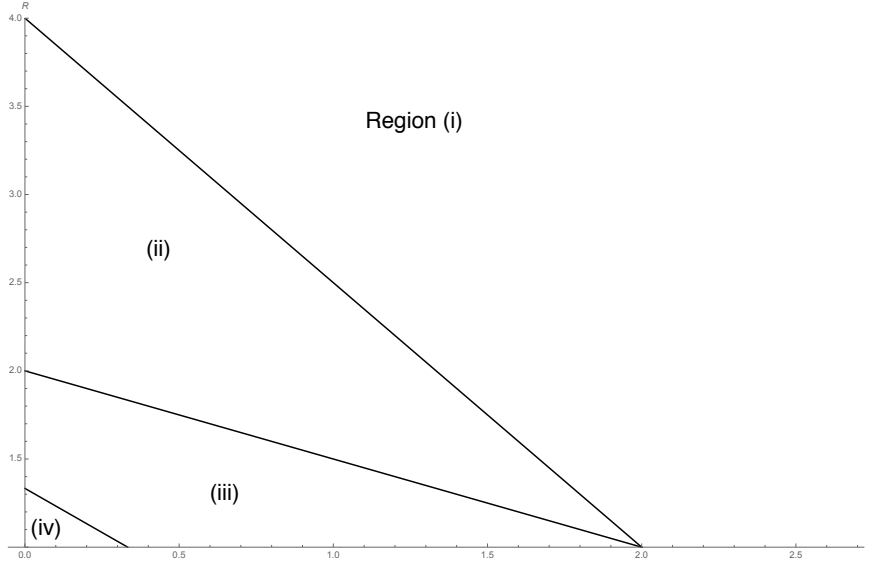}
\caption{The regions in the four parts of Theorem \ref{t:interval4}. Examples from the four cases are further illustrated in Figure \ref{f:RA1}. }
\label{f:RA0}
\end{center}
\end{figure}

\begin{figure}[h!]
\begin{center}
\includegraphics[width=.9\textwidth]{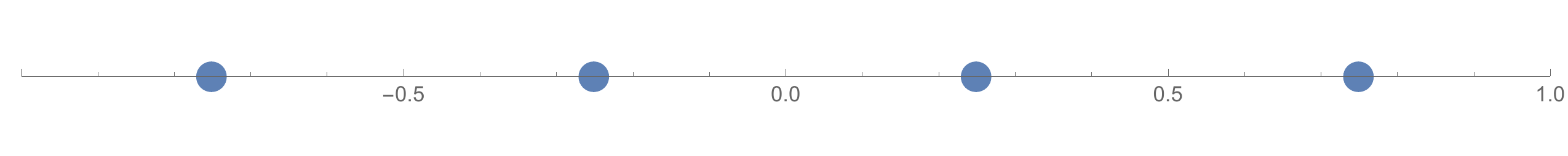}
\includegraphics[width=.9\textwidth]{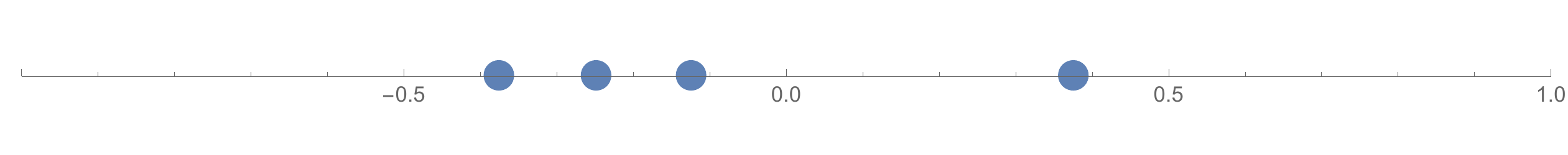}
\includegraphics[width=.9\textwidth]{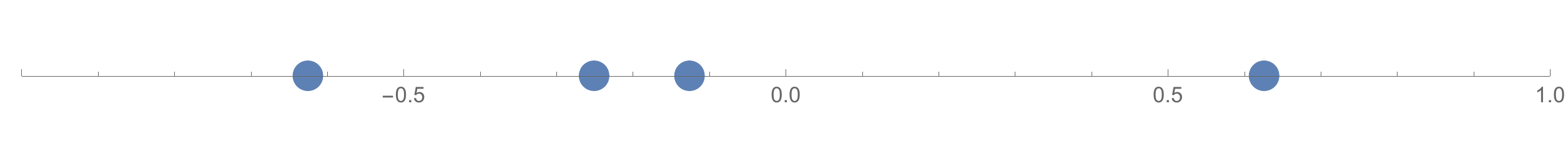}
\includegraphics[width=.9\textwidth]{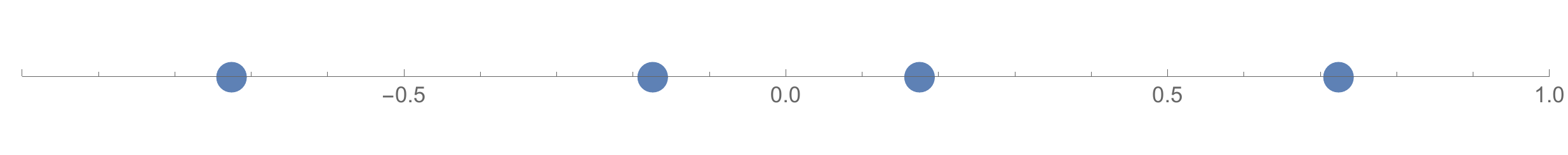}
\caption{An optimal point configuration for the following cases from Theorem \ref{t:interval4}, as illustrated in Figure \ref{f:RA0}: 
(i) $r=2=R$, 
(ii) $r=1/2$ and $R=5/2$, 
(iii) $r=1/2$ and $R=3/2$, and 
(iv)  $r=0.1$ and $R=1.1$.}
\label{f:RA1}
\end{center}
\end{figure}

\begin{thm}\label{t:interval4}
Suppose $\Omega=[-1,1]$ and $k(x,y)=f(|x-y|)$ for some completely monotone function $f$.  If $n=4$, then an admissible configuration exists if and only if $r\leq8/3$ and $R\geq1$.  In that case, the extremal configurations $\{x_1<x_2<x_3<x_4\}$ are given by
\begin{itemize}
\item[i)]  If $8\leq 3r+2R$, then $|x_1-x_2|=|x_2-x_3|=|x_3-x_4|=r/n$ and $x_1\leq-1+R/4$ and $x_4\geq1-R/4$
\item[ii)]  If $2r+4R\geq8>3r+2R$, then $x_1=-1+R/4$ and $x_4=1-R/4$ and $x_3-x_2=x_2-x_1=r/4$
\item[iii)]  If $8>3r+2R$ and $6R+r\geq8>2r+4R$, then $x_1=-1+R/4$ and $x_4=1-R/4$ and $x_4-x_3=2R/4$ and $x_3-x_2=r/4$
\item[iv)]  If $8>\max\{3r+2R,6R+r,2r+4R\}$, then $x_1=-1+R/4$ and $x_4=1-R/4$ and $x_4-x_3=x_2-x_1=2R/4$
\end{itemize}
\end{thm}

The regions in the four cases of Theorem \ref{t:interval4} are illustrated in Figure \ref{f:RA0}.   In Figure \ref{f:RA1}, we illustrate an optimal configuration for a choice of $r,R$ for each of the four cases in Theorem \ref{t:interval4}. Note that the optimal configurations for regions (ii) and (iii) break the symmetry of the interval.

\begin{rem} \label{r:HigherDim}
The higher dimensional case, \ie, \eqref{e:DiscOpt} for $d\geq2$ dimensional ball is more difficult. From Proposition \ref{p:ball}, intuitively we should pack the points as close as possible in the center of the ball. In two dimensions, this would be a triangular packing with a spacing given by $rn^{-1/d}$ with $d=2$. Away from the center region, we should put the points at the centers of an optimal covering where the  spacing is given by $Rn^{-1/d}$. Of course, these two configurations won't agree perfectly at the interface (``geometric frustration''), but we expect this gives an approximate solution in the limit as $n\to \infty$.
\end{rem}

\section{A computational method for \eqref{e:DensityFormulation}} \label{s:ex} 
The implicit relationship in  Proposition \ref{p:bang-bang} that characterizes the optimal density motivates the  rearrangement algorithm given in Algorithm \ref{a:ra}.  Here we alternatively apply the integral operator, $K$, defined in \eqref{eq:HS}, and threshold the result in such a way so that $\int_\Omega \rho = 1$.  We've stated Algorithm \ref{a:ra} assuming that $|\{ K\rho_s = \alpha \} |_d = 0$ for all $s\in \mathbb N$. 

\begin{prop}  \label{prop:AlgIncrease}
Let $\Omega$ and $k$ satisfy the assumptions in Section \ref{s:Intro}.  Let $\rho_s$,  for $s= 0, 1, \ldots$,  be the iterates
of Algorithm \ref{a:ra} and assume that $|\{ K\rho_s = \alpha \} |_d = 0$.
Then the sequence $E[\rho_s]$ is strictly increasing for non-stationary iterates. 
\end{prop}
\begin{proof}
By the strict convexity of $E$ and \eqref{eq:deriv}, for $\rho_{s+1} \neq \rho_s$,
$$
E(\rho_{s+1})- E(\rho_s) 
>  \langle K[\rho_s],  \rho_{s+1} - \rho_s \rangle 
=  \langle K[\rho_s ],  \rho_{s+1} \rangle -  \langle K[\rho_s ],  \rho_{s} \rangle . 
$$
The bathtub principle \citep[Theorem 1.14]{lieb2001analysis} shows that 
$$ \langle K[\rho_s ],  \rho_{s+1} \rangle  \geq   \langle K[\rho_s ],  \rho_{s} \rangle.  $$ 
This, in turn, implies that  $E(\rho_{s+1}) >  E(\rho_s)$ for all non-stationary iterations. 
\end{proof}
Non-stationary iterations of this algorithm have strictly increasing values, $E[\rho_s]$, so the sequence $\{\rho_s \}_{s=1}^\infty$ will have a limit point, but we have not proven that such points are optimal $\rho^\star$. However, for a discretization of the problem, there are only a finite number of $\{ \rho_-, \rho_+ \}$-valued functions, so  
Proposition \ref{prop:AlgIncrease} shows that  Algorithm \ref{a:ra} converges to a critical point in a finite number of iterations.  

The following proposition shows that the algorithm preserves symmetry: if $\Omega$ and the initial $\rho_0(x) \in A(\Omega,\rho_+, \rho_-)$ have a reflection symmetry, the algorithm can only converge to a critical point with the same symmetry. One example of this behavior can be illustrated for the disjoint union of two identical intervals considered in Section \ref{s:Intervals}. 
\begin{prop} Let $k(x,y) = f(|x-y|)$ for some $f\colon \mathbb R \to \mathbb R$. Suppose $\rho_0(x) \in A(\Omega,\rho_+, \rho_-)$ has the same reflection symmetry as $\Omega$. 
Assume that the iterates  $\{ \rho_s \}_s$ of Algorithm  \ref{a:ra} satisfy  $|\{ K\rho_s = \alpha \} |_d = 0$.
%Suppose also that any ambiguity in the choice of level set in the thresholding step of Algorithm  \ref{a:ra}  is resolved to preserve this symmetry. 
Then all iterates enjoy the same reflection symmetry. 
\end{prop}
\begin{proof} 
Let $r \colon \mathbb R^p \to \mathbb R^p$ denote a reflection over a line of symmetry with $r(\Omega) = \Omega$ and $\rho_0(r(x)) = \rho_0(x)$. 
We'll show that $\rho_1$ satisfies $\rho_1(r(x)) = \rho_1(x)$ a.e., which by induction proves the proposition. We compute for a.e. $x \in \Omega$,
\begin{align*}
K\rho( r(x)) = \int_{\Omega} f( | r(x) - y| ) \rho(y) dy 
= \int_{r(\Omega)} f( | x - r(y)| ) \rho(r(y)) dy 
= \int_{\Omega} f( | x - z| ) \rho(z) dz 
= K \rho(x). 
\end{align*}
\end{proof}

\begin{rem} \label{r:MBO}
Algorithm  \ref{a:ra} is very similar to the Merriman-Bence-Osher 
(MBO) diffusion-generated method \cite{MerrimanBenceOsher94} with the following differences:
(i) The ``diffusion step'' in MBO  (convolution with the heat kernel) has been replaced by a more general integral operator in \eqref{eq:HS} and 
(ii) the ``thresholding step'' in MBO is replaced by a volume preserving thresholding step as in \cite{ruuth2003simple}. In this context, the energy \eqref{e:E} can be viewed as the corresponding generalization of the Lyapunov function for MBO given in \cite{esedoglu2013}. 
\end{rem}

\begin{algorithm}[t]
\caption{\label{a:ra} A rearrangement algorithm for solving \eqref{e:DensityFormulation}. }
\KwData{ Initial guess $\rho_0(x) \in A(\Omega,\rho_+, \rho_-)$ and a convergence tolerance $\varepsilon > 0$.}
Set s=0.

\While{ $s\leq1$ or $\| \rho_s - \rho_{s-1} \|_{L^1(\Omega)} > \varepsilon$, }{
Define $\phi = K\rho_s$.  

Find the value $\alpha$ such that $ | \{ \phi \geq \alpha \} |_d  = \frac{|\Omega|_d^{-1} -\rho_-}{\rho_+ - \rho_-} |\Omega|_d$. 

Set
$ \rho_{s+1}(x) = \begin{cases}
\rho_+ & \phi(x) \geq \alpha  \\
\rho_- & \phi(x) < \alpha 
\end{cases} $. 

Set s = s+1.
}
\end{algorithm}

\subsection{Computational examples} \label{s:CompEx}
We implement  Algorithm \ref{a:ra}  in Matlab and consider several examples. In all of the following examples, the exponential kernel, $k(x,y) = \exp(-|x-y|)$, is used.

\subsection*{``Clover-shaped'' domain} Consider the ``clover shaped'' domain, $\Omega$, given in polar coordinates by 
$$
\Omega = \{ (r,\theta\} \colon r \leq 1 + 0.3\cdot\cos(4 \theta) \} \subset \mathbb R^2
$$
and an initial $\rho_0(x)$ which is given in the left panel of Figure \ref{f:RA}. The parameters are chosen so that $|\Omega_+|_d / |\Omega|_d = .25$.  Here white denotes $\rho_+$ and black denotes $\rho_-$.  For a $200\times200$ discretization, the iterations of Algorithm \ref{a:ra} become stationary in 5 iterations. Iterations 1 and 5 are also plotted in the center and right panels of Figure \ref{f:RA}.   At the stationary solution, $\Omega_+$ is a ball centered in $\Omega$.

\subsection*{Annulus} We next consider the annulus, given in polar coordinates by 
$$
\Omega_\alpha = \{ (r,\theta\} \colon \alpha \leq r \leq 1.2  \} \subset \mathbb R^2. 
$$
We'll consider varying the inner radius, $\alpha$. 
Consider $\alpha = 0.6$ and an initial $\rho_0(x)$ which is given in the top left panel of Figure \ref{f:RC}. The parameters are chosen so that $|\Omega_+|_d / |\Omega|_d = 0.1$.  Here white denotes $\rho_+$ and black denotes $\rho_-$.  For a $200\times200$ discretization, the iterations of Algorithm \ref{a:ra} become stationary in 48 iterations. Iterations 1 and 48 are also plotted in the top center and top right panels of Figure \ref{f:RC}.   At the stationary solution, $\Omega_+$ is a centered annulus. 

 Consider $\alpha = 0.7$ and an initial $\rho_0(x)$ which is given in the bottom left panel of Figure \ref{f:RC}. The parameters are chosen so that $|\Omega_+|_d / |\Omega|_d = 0.1$.  Here white denotes $\rho_+$ and black denotes $\rho_-$.  For a $200\times200$ discretization, the iterations of Algorithm \ref{a:ra} become stationary in 9 iterations. Iterations 1 and 9 are also plotted in the bottom center and bottom right panels of Figure \ref{f:RC}.   The optimal solution breaks the  symmetry of the annulus,  as in the explicit example from  Section \ref{s:Intervals},

\begin{figure}[t!]
\begin{center}
\includegraphics[width=.75\textwidth]{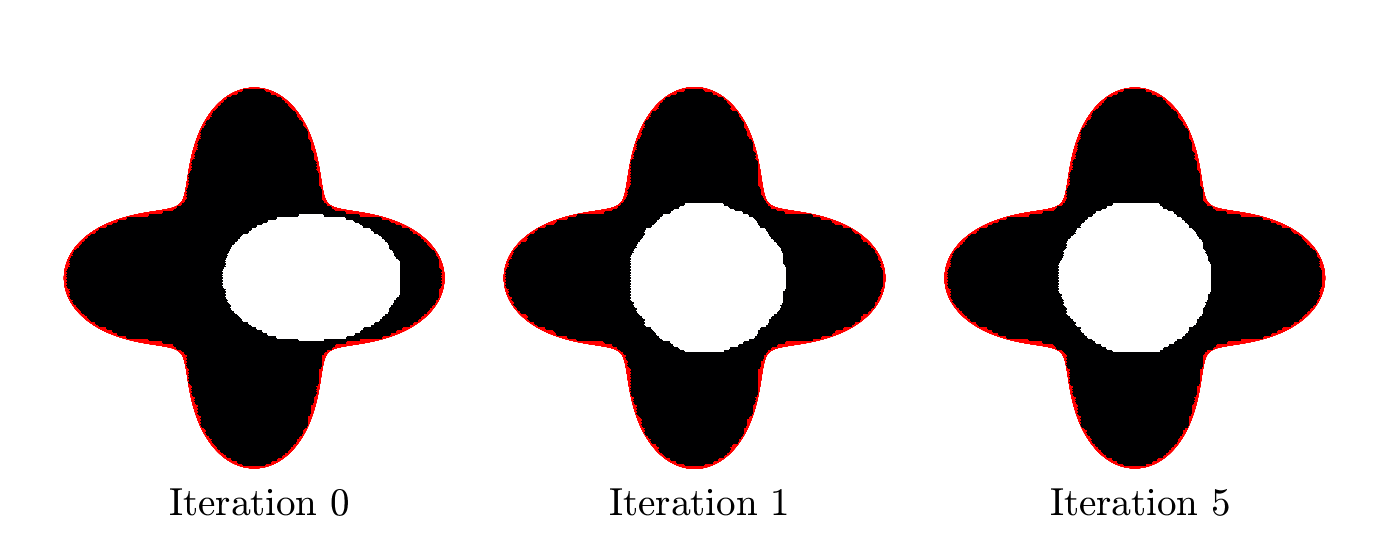}
\caption{An illustration of the rearrangement algorithm (Algorithm \ref{a:ra}) for a ``clover-shaped'' domain, outlined in red. The white region corresponds to $\Omega_+$ (where $\rho = \rho_+$) and the black region denotes $\Omega_-$ (where $\rho = \rho_-$). Algorithm \ref{a:ra} converges in 5 iterations with a final $\Omega_+$ that is a centered domain. See Section \ref{s:CompEx}.}
\label{f:RA}
\end{center}
\end{figure}

\begin{figure}[t!]
\begin{center}
\includegraphics[width=.75\textwidth]{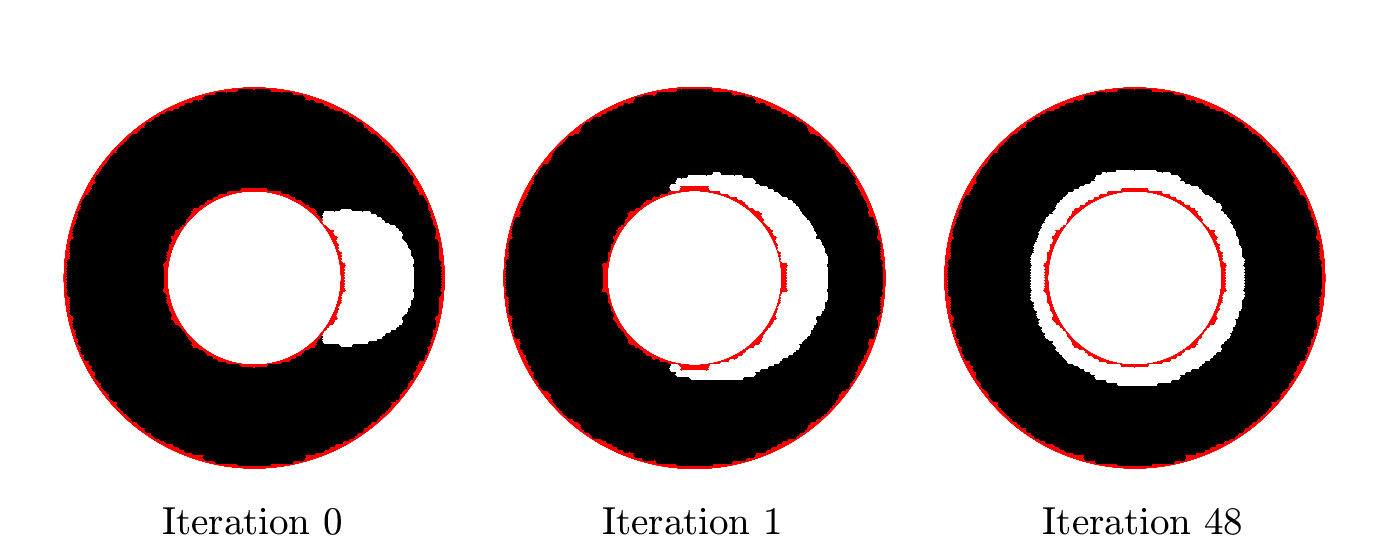}
\includegraphics[width=.75\textwidth]{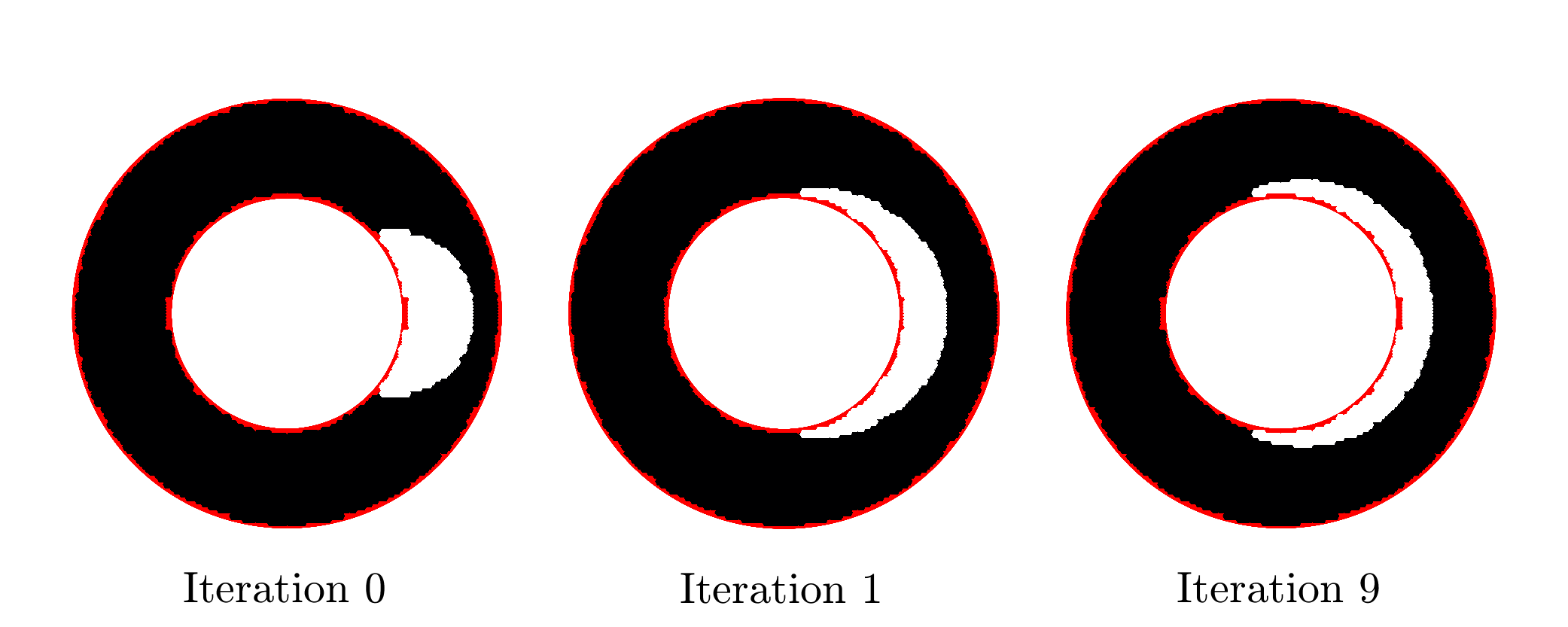}
\caption{An illustration of the rearrangement algorithm (Algorithm \ref{a:ra}) for an annulus, $\Omega_\alpha$. The white region corresponds to $\Omega_+$ (where $\rho = \rho_+$) and the black region denotes $\Omega_-$ (where $\rho = \rho_-$). 
{\bf (top)} For inner radius $\alpha = 0.6$ and this initial condition, Algorithm \ref{a:ra} converges in 48 iterations with a final $\Omega_+$ that is a centered annulus. 
{\bf (bottom)} For inner radius $\alpha = 0.7$ and this initial condition, Algorithm \ref{a:ra}  converges in 9 iterations with a final $\Omega_+$ that breaks the rotational symmetry of the annulus. 
See Section \ref{s:CompEx}.}
\label{f:RC}
\end{center}
\end{figure}

\begin{figure}[t!]
\begin{center}
\includegraphics[width=.8\textwidth,trim= 0 0 0 30,clip]{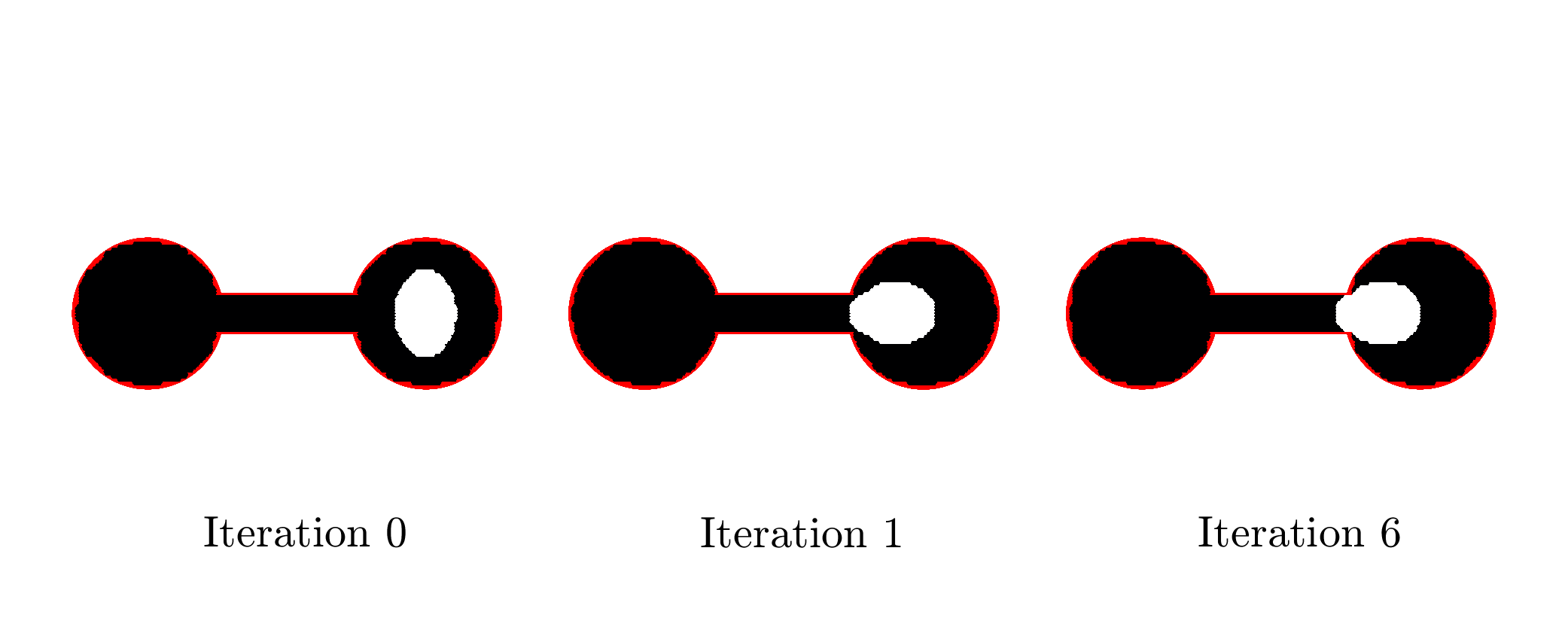}
\includegraphics[width=.8\textwidth,trim= 0 0 0 30,clip]{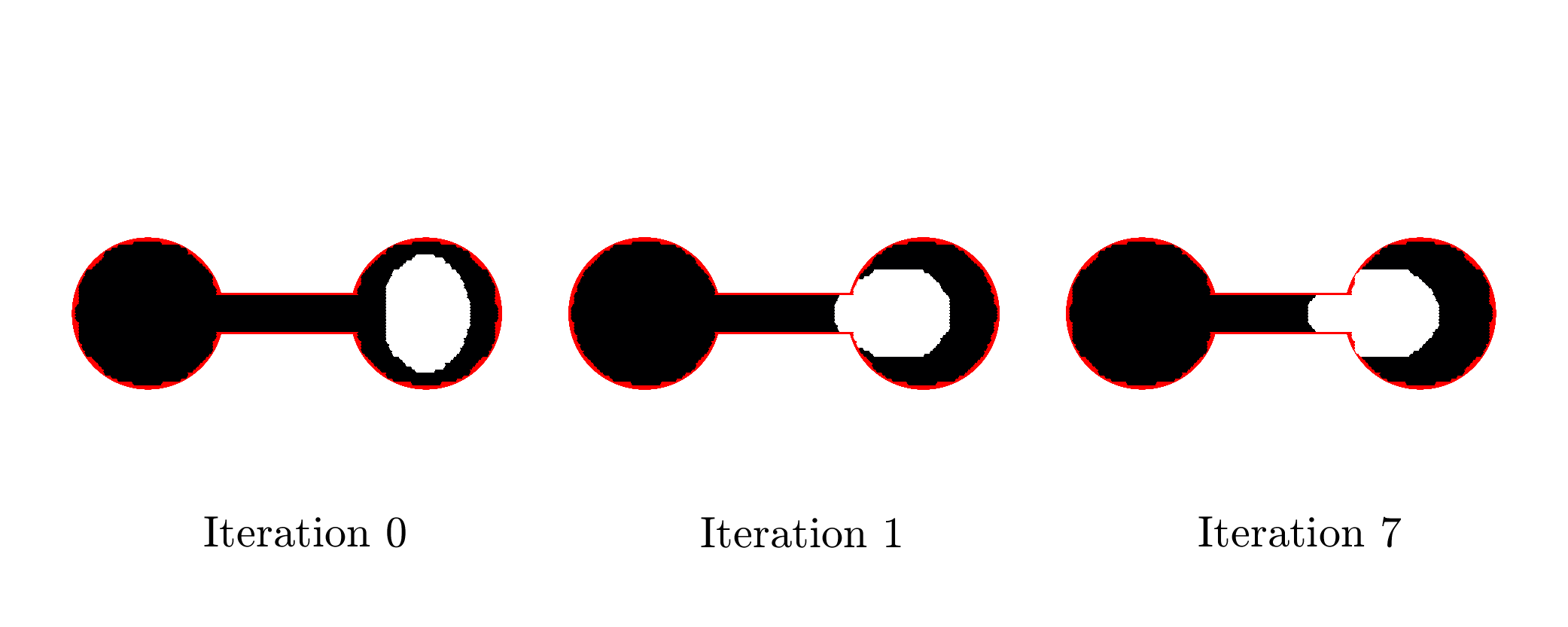}
\includegraphics[width=.8\textwidth,trim= 0 0 0 30,clip]{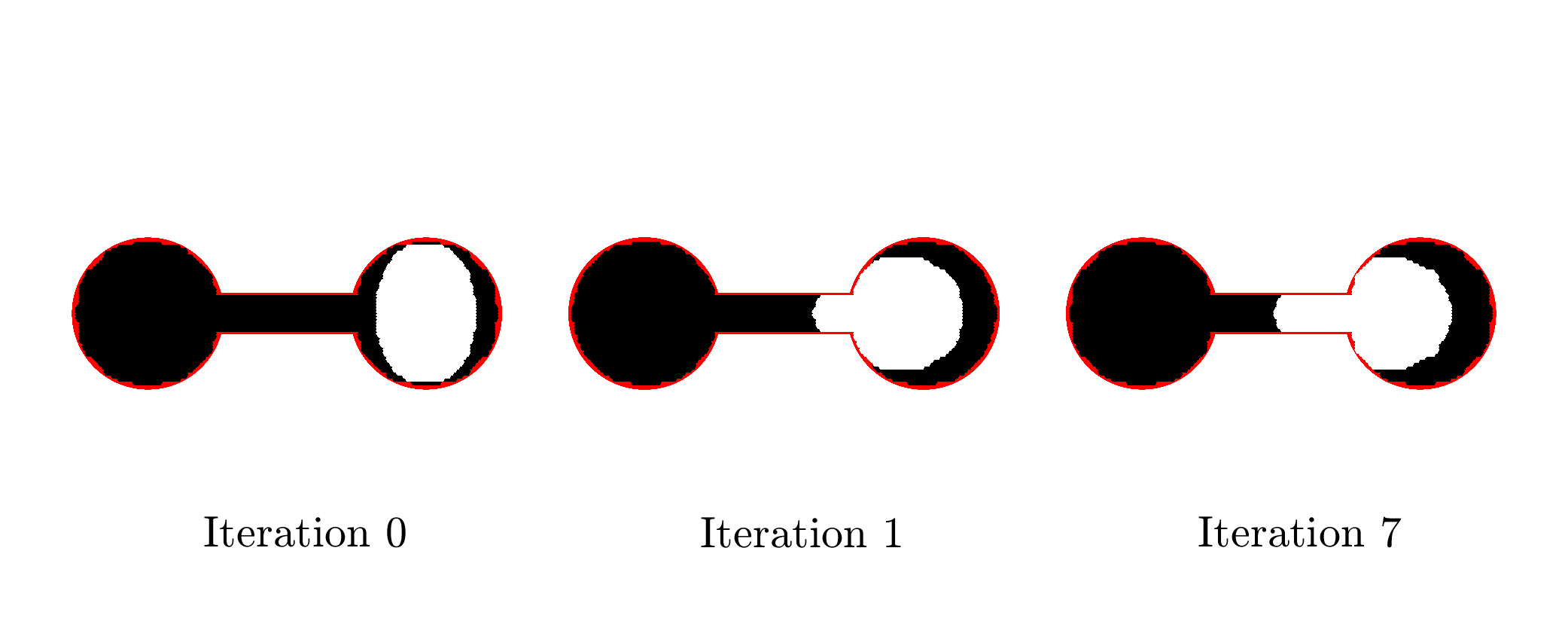}
\includegraphics[width=.8\textwidth,trim= 0 0 0 30,clip]{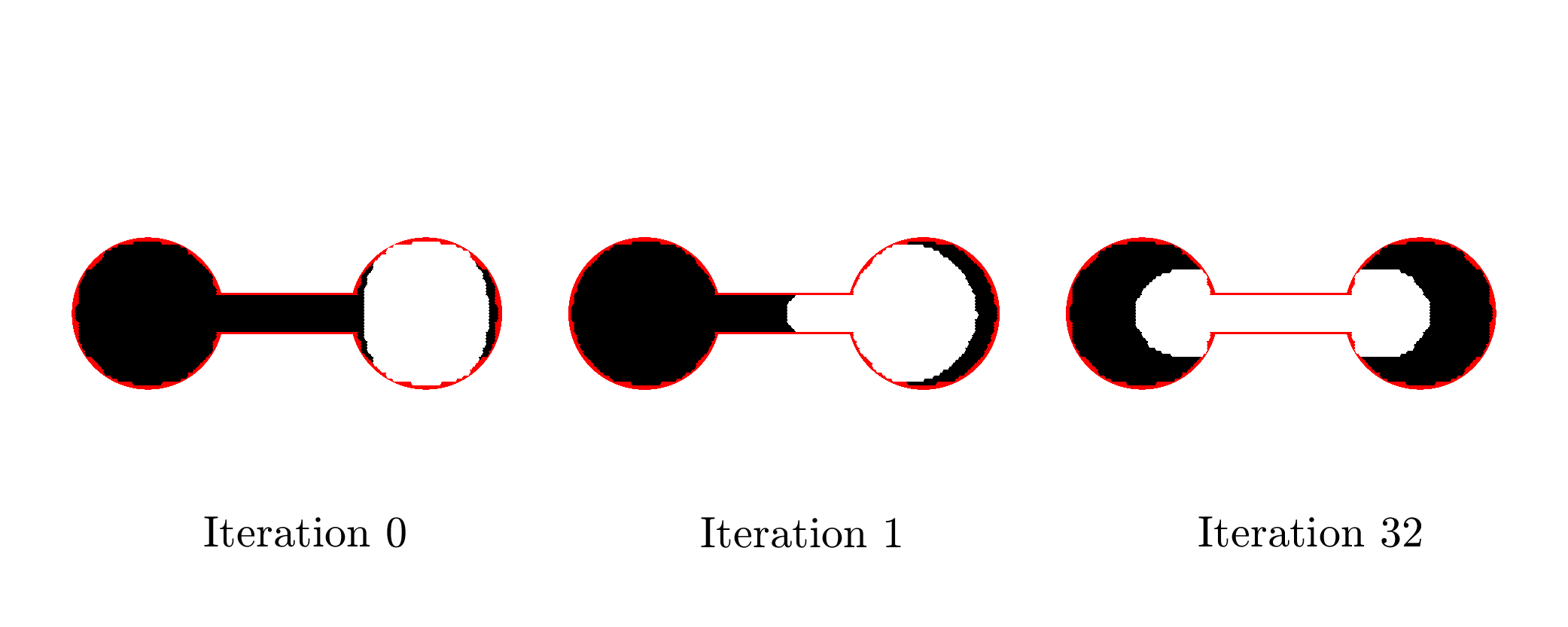}
\caption{An illustration of the rearrangement algorithm (Algorithm \ref{a:ra}) for a ``dumbbell-shaped'' domain. The white region corresponds to $\Omega_+$ (where $\rho = \rho_+$) and the black region denotes $\Omega_-$ (where $\rho = \rho_-$). The initial conditions (left) are chosen so that $|\Omega_+|_d/|\Omega|_d = 0.1, \ 0.2, \ 0.3$ and 0.4. The first iterations are shown in the center panel and the stationary states are shown in the right panel. Note that symmetry is broken in the top three solutions.  See Section \ref{s:CompEx}. }
\label{f:RA2}
\end{center}
\end{figure}

\subsection*{``Dumbbell-shaped'' domain}
Finally, we next consider the ``dumbbell shaped'' domain, $\Omega$, given by 
$$
B_{0.5}(1,0) \ \cup \  B_{0.5}(-1,0) \ \cup \ [-1,1] \times [-0.1,0.1] .
$$
In Figure \ref{f:RA2}, for a $200\times 200$ discretization,  we illustrate the evolution of $\rho$ under the rearrangement algorithm (Algorithm \ref{f:RA}) with initial condition such that 
 $|\Omega_+|_d/|\Omega|_d = 0.1, \ 0.2, \ 0.3$ and 0.4. As in the explicit example from  Section \ref{s:Intervals}, the solution breaks symmetry.

\section{Discussion} \label{s:disc}
In this paper, we considered the extremal pointset configuration problem \eqref{e:DiscOpt} of maximizing a kernel-based energy subject to geometric constraints.  We also formulated an extremal density problem \eqref{e:DensityFormulation} which we showed to be related to the pointset configuration problem in the limit as the number of points tends to infinity.  For both problems, we were able to show that extremal solutions exist under the appropriate hypotheses.  For the density problem, we were also able to deduce several important properties of the extremal density, such as the bang-bang property.  We explored several examples in great detail and provided an especially detailed analysis in the case of a sphere or a ball, where rearrangement inequalities allowed us to precisely describe the extremal solutions to the density problem.  In the general case, the optimal solution may not be unique or share the symmetries of the domain.  Our observations lead us to make a conjecture for a sufficient condition that implies uniqueness.  We concluded by developing a computational method for the density problem that is very similar to the Merriman-Bence-Osher (MBO) diffusion-generated method that we proved to be increasing for all non-stationary iterations.  The method was applied to study several additional example sets.

Our analysis did not provide an algorithmic solution to the pointset problem \eqref{e:DiscOpt} and one could ask if such an algorithm can be easily obtained (for $n$ large) from the solution of the continuous problem \eqref{e:DensityFormulation}. In Section \ref{s:DiscSphere}, on the sphere, we performed some calculations suggesting that this is possible.  More generally, we would like to claim that given a solution to \eqref{e:DensityFormulation} that defines a partition, $\Omega = \Omega_+ \sqcup \Omega_-$, one should be able to approximately solve \eqref{e:DiscOpt} by placing
$n_+ = n \rho_+ |\Omega_+|_d  $ points in a best packing configuration in the region $\Omega_+$ and $n_- = n \rho_- |\Omega_-|_d $ points in a best covering configuration in the region $\Omega_-$.  Some small modification of the configuration would be required near the interface between $\Omega_+$ and $\Omega_-$ to satisfy the constraints in \eqref{e:DiscOpt}; see Remark \ref{r:HigherDim}. From the computational experiments in Section \ref{s:ex}, we suspect that the interface between the regions $\Omega_+$ and $\Omega_-$ is very regular, so it might be possible to make this argument precise.  One substantial obstacle to the implementation of this algorithm is that best packing and best covering configurations are difficult to obtain (or even approximate), especially in high dimensions.  An interesting problem for future research would be to find precise solutions to \eqref{e:DiscOpt} for certain sets of interest and small values of $n$ as was done in Theorem \ref{t:interval4} for the interval when $n=4$.

\smallskip

Another possible approach to investigating the relationship between the  the discrete \eqref{e:DiscOpt} and continuous \eqref{e:DensityFormulation} problems considered here would be to study pointset configurations that arise naturally from other problems.  For instance, instead of using best-packing configurations and best-covering configuration as described above, one could use pointset configurations that minimize a Riesz energy or cubature nodes.  An extensive list of interesting configurations on $S^2$ is provided in \cite{michaels16}.  We note that i.i.d. random pointset configurations should not be considered for Problem \eqref{e:DiscOpt}.  Indeed, it is known that i.i.d. uniformly sampled points are not expected to be admissible for \eqref{e:DiscOpt}. Namely, \eqref{e:DiscOpt2} is violated since the expected separation distance is proportional to $n^{-2/d}$ \cite{cai2013distributions,brauchart2015random}.  In a similar vein, it would be interesting to consider how solutions to \eqref{e:DensityFormulation} compare to maximizers of other objectives that are convex in $\rho$, such as the spectral objectives considered in \cite{OM16}.

\smallskip

While we explored properties of extremal solutions to \eqref{e:DensityFormulation}, there is still much more we would like to know about the regions $B_{\pm}$ (see Conjecture \ref{c:convex}).  Another interesting problem to explore would be to find the right hypotheses on $\Omega$ to ensure that $B_+$ is connected.  In Section \ref{s:ex}, we developed a rearrangement algorithm for finding critical points of \eqref{e:DensityFormulation}. As commented in Remark \ref{r:MBO}, for the particular case that the kernel is $k(x,y) = (4 \pi \tau)^{-d/2} \exp(- |x-y|^2 / 4\tau)$ with $\tau >0$, this is similar  to MBO diffusion generated method with a volume constraint.  In the limit as $\tau \to 0$, the MBO evolution evolves according to mean-curvature flow \cite{Evans93} and Algorithm \ref{a:ra} thus minimizes the volume of the boundary between $\Omega_+$ and $\Omega_-$. That is, it appears that the interface between the sets where $\rho=\rho_+$ and $\rho=\rho_-$ is a minimal surface.  It would be useful to have a rigorous theorem to this effect.

It is tempting to think that the optimal $\rho$ could be associated with level sets of some function, \eg, the potential, $V(x) = \int_\Omega k(x,y) dy$ or  the principal eigenfunction of $K$. However, Figure \ref{f:RA2} provides a counterexample for this in the non-convex case. Looking at the middle-right and bottom-right panels of the figure, the sets are not subsets of one another  and therefore cannot both be the level sets of the same function.

%\clearpage
\printbibliography

\end{document}